\documentclass[11pt,reqno]{amsart}
\usepackage{amsmath,amssymb,amsthm,enumerate,natbib,color,bbm,ifthen,graphicx,overpic}
\usepackage[latin1]{inputenc} 
\def\nset{{\mathbb{N}}}
\def\rset{\mathbb R}

\def\eqsp{\;}

\newcommand{\as}{\text{a.s.} }

\newcommand{\seq}[1]{\left<#1\right>}
\newcommand{\un}{\ensuremath{\mathbbm{1}}}
\newcommand{\eqdef}{\ensuremath{\stackrel{\mathrm{def}}{=}}}

\def\Xset{\mathsf{X}} 
\def\Xsigma{\mathcal{X}} 
\def\Tsigma{\mathcal{B}(\Theta)} 
\def\F{\mathcal{F}} 
\def\Upsi{\Upsilon}

\def\PP{\mathbb{P}} 
\def\PE{\mathbb{E}} 
\def\cPP{\check{\mathbb{P}}} 
\def\cPE{\check{\mathbb{E}}} 

\def\W{\mathcal{W}}
\def\L{\mathcal{L}} 

\def\tv{\mathrm{TV}}

\def\tauout{\stackrel{\leftarrow}{\tau}} 
\def\compact{\mathsf{K}}  
\def\Cset{\mathcal{C}} 

\newlength{\noteWidth}
\newtheorem{theo}{Theorem}[section]
\newtheorem{lemma}[theo]{Lemma}
\newtheorem{coro}[theo]{Corollary}
\newtheorem{prop}[theo]{Proposition}

\newtheorem{algo}{Algorithm}[section]
\theoremstyle{remark}
\newtheorem{rem}{Remark}

\newcounter{hypoconbis}
\newcounter{saveconbis}
\newcommand\debutA{\begin{list} {\textbf{A\arabic{hypoconbis}}}{\usecounter{hypoconbis}}\setcounter{hypoconbis}{\value{saveconbis}}}
\newcommand\finA{\end{list}\setcounter{saveconbis}{\value{hypoconbis}}}

\newcounter{hypoconbisp}
\newcounter{saveconbisp}
\newcommand\debutAp{\begin{list} {\textbf{A\arabic{hypoconbisp}'}}{\usecounter{hypoconbisp}}\setcounter{hypoconbisp}{\value{saveconbisp}}}
\newcommand\finAp{\end{list}\setcounter{saveconbisp}{\value{hypoconbisp}}}

\newcounter{hypocom}
\newcounter{savecom}
\newcommand{\debutB}{\begin{list}{\textbf{B\arabic{hypocom}}}{\usecounter{hypocom}}\setcounter{hypocom}{\value{savecom}}}
\newcommand{\finB}{\end{list}\setcounter{savecom}{\value{hypocom}}}

\newcounter{hypocomp}
\newcounter{savecomp}
\newcommand{\debutBp}{\begin{list}{\textbf{B\arabic{hypocomp}'}}{\usecounter{hypocomp}}\setcounter{hypocomp}{\value{savecomp}}}
\newcommand{\finBp}{\end{list}\setcounter{savecomp}{\value{hypocomp}}}

\newcounter{hypostab}
\newcounter{savestab}
\newcommand{\debutC}{\begin{list}{\textbf{C\arabic{hypostab}}}{\usecounter{hypostab}}\setcounter{hypostab}{\value{savestab}}}
\newcommand{\finC}{\end{list}\setcounter{savestab}{\value{hypostab}}}

\newcounter{hypodist}
\newcounter{savedist}
\newcommand{\debutD}{\begin{list}{\textbf{D\arabic{hypodist}}}{\usecounter{hypodist}}\setcounter{hypodist}{\value{savedist}}}
\newcommand{\finD}{\end{list}\setcounter{savedist}{\value{hypodist}}}

\begin{document}

\title[Subgeometric Adaptive MCMC]{Limit theorems for some adaptive MCMC algorithms with subgeometric kernels: Part II}

\author[Y. F. Atchad\'e]{Yves F. Atchad\'e} \thanks{ Y. Atchad\'e: University of Michigan, 1085 South University, Ann Arbor,
  48109, MI, United States. {\em E-mail address} yvesa@umich.edu}

\author[G. Fort]{ Gersende Fort} \thanks{G. Fort: LTCI, CNRS-TELECOM ParisTech,
  46 rue Barrault, 75634 Paris Cedex 13, France. {\em E-mail address}
  gfort@tsi.enst.fr}

\thanks{This work is partly supported by NSF grant DMS 0906631 and by the french National Research Agency
  (ANR) program ANR-05-BLAN-0299.}



\maketitle

\begin{center} Aug. 2009 \end{center}

\begin{abstract}

We prove a central limit theorem for a general class of adaptive Markov Chain Monte Carlo algorithms driven by sub-geometrically ergodic Markov kernels. We discuss in detail the special case of stochastic approximation. We use the result to analyze the asymptotic behavior of an adaptive version of the Metropolis Adjusted Langevin algorithm with a heavy tailed target density.
\end{abstract}

\section{Introduction}
This work is a sequel of \cite{atchadeetfort08} and develops central limit theorems for adaptive MCMC (AMCMC) algorithms. Previous works on the subject include \cite{andrieuetal06} and \cite{saksmanvihola09} where central limit theorems are proved for certain AMCMC algorithms driven by geometrically ergodic Markov kernels. There is a need to understand the sub-geometric case. Indeed, many Markov kernels routinely used in practice are not geometrically ergodic. For example, if the target distribution of interest has heavy tails, then the Random Walk Metropolis algorithm (RWMA) and the Metropolis Adjusted Langevin algorithm (MALA) result in sub-geometric Markov kernels (\cite{jarneretroberts201}).

We consider adaptive MCMC algorithms driven by Markov kernels $\{P_\theta,\;\theta\in\Theta\}$ such that each kernel $P_\theta$ enjoys a polynomial rate of convergence towards $\pi$ and  satisfies a drift condition of the form $P_\theta V\leq V -cV^{1-\alpha} +b$ for some $\alpha\in (0,1]$ (uniformly in $\theta$ over compact sets). We obtain a central limit theorem when $\alpha<1/2$ under some additional stability conditions. This result is very close to what can be proved for Markov chains under similar conditions. Indeed, it is known (\cite{jarneretroberts101}) that irreducible and aperiodic Markov chains for which the drift condition $P V\leq V -c V^{1-\alpha} +b\un_{\Cset}$ hold for some small set $\Cset$ satisfy a central limit theorem when $\alpha\leq 1/2$.  The slight loss of efficiency in our case ($\alpha<1/2$ versus $\alpha\leq 1/2$) is typical of martingale approximation-based proofs. The proof of the central limit theorem is based on a martingale approximation technique initiated by \cite{kipnisetvaradhan86} and \cite{mw00}. The method is a Poisson equation-type method but where the Poisson's kernel is replaced by a more general resolvent kernel. We have used a variant of the same technique in \cite{atchadeetfort08} to study the strong law of large numbers for AMCMC.

Adaptive MCMC has been studied in a number of recent papers. Beside the above mentioned papers, results related to the convergence of marginal distributions and the law of large numbers can be found e.g. in (\cite{rosenthaletroberts05, bai:2008}). For specific examples and a review of the methodological developments, see e.g. \cite{robertsetrosenthal06, andrieu:thoms:2008,atchadeetal09}.

The rest of the paper is organized as follows. The main CLT result is presented in Section \ref{sec:CLTresults}. Adaptive MCMC driven by stochastic approximation is considered in Section \ref{sec:SA}. To illustrate, we apply our theory to an adaptive version of the Metropolis adjusted Langevin algorithm (MALA) with a heavy tailed target distribution (Section \ref{sec:ex}). Most of the proofs are postponed to Section \ref{sec:proofs}.

\section{Statement of the results}\label{sec:mainres}
\subsection{Notations}
We start with some notations that will be used through the paper. For a transition kernel $P$ on a measurable general state space
$(\mathbb{T},\mathcal{B}(\mathbb{T}))$, denote by $P^n$, $n\geq 0$, its $n$-th
iterate defined as
\[
P^0(x,A) \eqdef \delta_x(A) \eqsp, \qquad \qquad P^{n+1}(x,A) \eqdef \int
P(x,dy ) P^n(y,A) \eqsp, \quad n \geq 0 \eqsp;
\]
$\delta_x(dt)$ stands for the Dirac mass at $\{x\}$. $P^n$ is a transition kernel
on $(\mathbb{T},\mathcal{B}(\mathbb{T}))$ that acts both on bounded measurable
functions $f$ on $\mathbb{T}$ and on $\sigma$-finite measures $\mu$ on
$(\mathbb{T},\mathcal{B}(\mathbb{T}))$ via $P^nf(\cdot) \eqdef \int
P^n(\cdot,dy) f(y)$ and $\mu P^n(\cdot) \eqdef \int \mu(dx) P^n(x, \cdot)$.

If $V: \mathbb{T}\to [1, +\infty)$ is a function, the $V$-norm of a function
$f: \mathbb{T}\to \rset$ is defined as $|f|_V \eqdef \sup_{\mathbb{T}} |f|
/V$.  When $V=1$, this is the supremum norm. The set of functions with finite
$V$-norm is denoted by $\L_V$.

If $\mu$ is a signed measure on a measurable space
$(\mathbb{T},\mathcal{B}(\mathbb{T}))$, the total variation norm $\| \mu
\|_{\tv}$ is defined as
\[
\| \mu \|_{\tv} \eqdef \sup_{\{f, |f|_1 \leq 1 \}} | \mu(f)| = 2 \; \sup_{A \in
  \mathcal{B}(\mathbb{T})}|\mu(A)|= \sup_{A \in \mathcal{B}(\mathbb{T})} \mu(A)
- \inf_{A \in \mathcal{B}(\mathbb{T})} \mu(A) \eqsp;
\]
and the $V$-norm, where $V : \mathbb{T} \to [1, +\infty)$ is a function, is
defined as $\| \mu \|_{V} \eqdef \sup_{\{g, |g|_V \leq 1 \}} |\mu(g)|$. Observe that $\|\cdot\|_{\tv}$ corresponds to $\|\cdot\|_{V}$ with $V\equiv 1$.

In the Euclidean space $\rset^n$, we use $\seq{a,b}$ to denote the inner product and $|a|\eqdef\sqrt{\seq{a,a}}$ the Euclidean norm. We denote $\rset$ the set of real numbers and $\nset$ the set of nonnegative integers.


\subsection{Adaptive MCMC: definition}\label{sec:defalgo}
Let $\Xset$ be a general state space resp. endowed with a countably
generated $\sigma$-field $\Xsigma$. Let $\Theta$ be an open subspace of $\rset^q$ the $q$-dimensional Euclidean space and $\Tsigma$ is its Borel $\sigma$-algebra. Let $\{P_\theta, \theta \in \Theta \}$ be a family of Markov transition kernels on $(\Xset,\Xsigma)$ such
that for any $(x,A) \in \Xset \times \Xsigma$, $\theta \mapsto P_\theta(x,A)$
is measurable. We assume that for any $\theta\in\Theta$, the Markov kernel $P_\theta$ admits an invariant distribution $\pi$.
Let $\{\compact_n, n\geq 0\}$ be a family of nonempty compact subspaces of $\Theta$ such that $\compact_n \subseteq\compact_{n+1}$. Let
$\Pi : \Xset\times \Theta \to \Xset_0\times \Theta_0$ be a measurable function, the so-called re-projection function, where $\Xset_0\times\Theta_0$ is some measurable subset of $\Xset\times \Theta$. We assume that $\Pi(x,\theta)=(x,\theta)$ if $\theta\in\Theta_0$. For an integer $k\geq 0$ we define $\Pi_k(x,\theta)=\Pi(x,\theta)$ if $k=0$ and $\Pi_k(x,\theta)=(x,\theta)$ if $k\geq 1$. Let $\bar R(n;\cdot,\cdot):\; (\Xset\times\Theta)\times (\Xsigma\times\Tsigma)\to [0,1]$ a sequence of Markov kernels on $\Xset\times\Theta$ with the following property. For any $n\geq 0$, $A\in\Xsigma$, $(x,\theta)\in\Xset\times \Theta$
 \begin{equation}\label{adapt1}
 \bar R\left(n;(x,\theta),A\times\Theta\right)=P_\theta(x,A).\end{equation}

In most cases in practice, the adaptation is driven by stochastic approximation. One such example of stochastic approximation is obtained by taking $\bar R(n;\cdot,\cdot)$ of the form $\bar R\left(n;(x,\theta),(dx',d\theta')\right)=P_\theta(x,dx')\delta_{\theta+\gamma_n\Upsi_\theta(x')}(d\theta')$. But the main example of stochastic approximation considered in this paper is
\begin{equation*}
\bar R\left(n;(x,\theta),(dx',d\theta')\right)=\int q^{(1)}_\theta(x,dy)q_\theta^{(2)}\left((x,y),dx'\right)\delta_{\theta+\gamma_n\Phi_\theta(x,y)}(d\theta').\end{equation*}
where $q^{(1)}_\theta$ and $q^{(2)}_\theta$ are Markov kernels. Obviously, in order for (\ref{adapt1}) to hold, these kernels ought to satisfy the constraint
\[\int q^{(1)}_\theta(x,dy)q_\theta^{(2)}\left((x,y),dx'\right)=P_{\theta}(x,dx').\]

Throughout the paper and without further mention, we assume that (\ref{adapt1}) hold. We are interested in the Markov chain $\{(X_n,\theta_n,\nu_n,\xi_n),\;n\geq 0\}$ define on $\Xset\times\Theta\times\nset\times\nset$ with transition kernel $\bar P$,
\begin{multline}\label{eq:tk2}
  \bar P \left( (x,\theta,\nu,\xi), (dx', d\theta', d\nu',d \xi')\right)\eqdef \bar R\left(\nu+\xi;\Pi_\xi(x,\theta),(dx',d\theta')\right)\\
  \times\left(\un_{\{\theta'\in\compact_\nu\}}\delta_{\nu}(d\nu') \delta_{\xi+1}(d\xi')+ \un_{\{\theta'\notin\compact_\nu\}} \delta_{\nu+1}(d\nu')\delta_{0}(d\xi')\right).
\end{multline}

\vspace{3cm}
Algorithmically, this Markov chain can be described as follows.
\begin{algo}\label{algo1}
Given $(X_n,\theta_n,\nu_n,\xi_n)$:
\begin{description}
\item [a] generate $(X_{n+1},\theta_{n+1})\sim \bar R\left(\nu_n+\xi_n;\Pi_{\xi_n}(X_n,\theta_n),\cdot\right)$;
 \item [b] if $\theta_{n+1}\in \compact_{\nu_n}$ then set $\nu_{n+1}=\nu_n$, $\xi_{n+1}=\xi_n+1$,
\item [c] if $\theta_{n+1}\notin \compact_{\nu_n}$ then set $\nu_{n+1}=\nu_n+1$ and $\xi_{n+1}=0$.
\end{description}
\end{algo}

We denote by $\cPP_{x,\theta,\nu,\xi}$ and $\cPE_{x,\theta,\nu,\xi}$ the
probability and expectation operator when the initial distribution of the Markov chain is $\delta_{(x,\theta,\nu,\xi)}$. Throughout the paper, we will assume that the initial state of the process is fixed to ($x_0,\theta_0,0,0)$ for some arbitrary element $(x_0,\theta_0)\in\Xset_0\times \Theta_0$ and we will systematically write $\cPP$ and $\cPE$ instead of $\cPP_{x_0,\theta_0,0,0}$ and $\cPE_{x_0,\theta_0,0,0}$ respectively.

\begin{rem}\label{remdefalgo}
Algorithm \ref{algo1} is fairly general and encompasses the two main strategies used in practice to control the adaptation parameter.
\begin{enumerate}
\item For example, one obtains the framework  of re-projections on randomly varying compact sets developed in (\cite{AMP05,andrieuetal06}) by taking $\{\compact_n, n\geq 0\}$ such that $\Theta = \bigcup_n \compact_n$, $\Theta_0\subseteq\compact_0$  and $\compact_n \subset \mbox{int}(\compact_{n+1})$, where $\mbox{int}(A)$ is the interior of $A$.
\item But we can also set $\Theta_0=\compact_k=\compact$ for all $k\geq 0$ for some compact subset $\compact$ of $\Theta$. And we then obtain another commonly used approach where the re-projection is done on a fixed compact set $\compact$. See  e.g. \cite{atchadeetrosenthal03}.
\end{enumerate}
\end{rem}

Let $\{\check\F_n,\;n\geq 0\}$ denote the natural filtration of the Markov chain $\{(X_n,\theta_n,\nu_n,\xi_n),\;n\geq 0\}$. It is easy to compute using (\ref{adapt1}) that for any bounded measurable function $f:\Xset\to\rset$,
\begin{equation}\label{adapt2}
\cPE\left(f(X_{n+1})\vert \check\F_n\right)\un_{\{\xi_n>0\}}=P_{\theta_n}f(X_n),\;\;\cPP-\mbox{a.s.}\end{equation}
Equation (\ref{adapt2}) together with the strong Markov property are the two main properties of the process $\{(X_n,\theta_n,\nu_n,\xi_n),\;n\geq 0\}$ that will used in the sequel.

We now introduce another stochastic process closely related to the adaptive chain defined above. For $l\geq 0$ an integer, we consider the nonhomogeneous Markov chain $\{(\tilde X_n,\tilde \theta_n),\;n\geq 0\}$ with initial distribution $\delta_{x,\theta}$ and sequence of transition Markov kernels
\[P_l\left(n;(x_1,\theta_1),(dx',d\theta')\right)=\bar R\left(l+n;(x_1,\theta_1),(dx',d\theta')\right).\]
Its distribution and expectation operator are denoted respectively by $\PP_{x,\theta}^{(l)}$ and $\PE_{x,\theta}^{(l)}$. We will denote $\{\F_n,\;n\geq 0\}$ its natural filtration (for convenience in the notations, we omit its dependence on $(x,\theta,l)$). Again it follows from (\ref{adapt1}) that for any bounded measurable function $f:\Xset\to\rset$,
\begin{equation}\label{adapt22}
\PE_{x,\theta}^{(l)}\left(f(\tilde X_{n+1})\vert \F_n\right)=P_{\tilde \theta_n}f(\tilde X_n),\;\;\PP_{x,\theta}^{(l)}-\mbox{a.s.}\end{equation}
For $\compact$ a compact subset of $\Theta$, we define the stopping time $\tauout_\compact$ (wrt the nonhomogeneous Markov chain $\{(\tilde X_n,\tilde \theta_n),\;n\geq 0\}$) as
\[\tauout_\compact=\inf\{k\geq 1:\; \tilde \theta_k\notin \compact\},\]
with the usual convention that $\inf\emptyset=\infty$. Clearly the two processes defined above are closely related. We will refer to $\{(\tilde X_n,\tilde \theta_n),\;n\geq 0\}$ as the \textit{re-projection free process}. The general strategy that we adopt to study the Markov chain $\{(X_n,\theta_n,\nu_n,\xi_n),\;n\geq 0\}$ (a strategy borrowed from \cite{AMP05}) consists in first studying the re-projection free process $\{(\tilde X_n,\tilde \theta_n),\;n\geq 0\}$ and showing that the former process inherits the limit behavior of the latter.

\subsection{General results}\label{sec:CLTresults}
The main assumption of the paper is the following.

\debutA
\item \label{A1} There exist $\alpha\in (0,1]$, and a measurable function $V:\; \Xset\to [1,\infty)$, $\sup_{x\in\Xset_0}V(x)<\infty$ with the following properties. For any compact subset $\compact$ of $\Theta$, there exists $b,c\in (0,\infty)$ (that depend on $\compact$) such that for any $(x,\theta) \in \Xset\times\compact$,
\begin{equation}\label{drift} P_\theta V(x)\leq V(x)-cV^{1-\alpha}(x)+b\end{equation}
and for any $\beta\in [0,1-\alpha]$, $\kappa\in [0,\alpha^{-1}(1-\beta)-1]$, there exists $C=C(V,\kappa,\beta,\compact)$ such that
\begin{equation}\label{rateconv}
  (n+1)^{\kappa} \; \| P_\theta^n(x,\cdot) - \pi(\cdot) \|_{V^\beta} \leq C \ V^{\beta +
  \alpha \kappa}(x),\;\;\;n\geq 0.
\end{equation}
\finA

Notice that (\ref{drift}) implies that $\pi(V^{1-\alpha})<\infty$. We will also assume that the number of re-projection is finite.

\debutA
\item \label{A2}
\begin{equation}\cPP\left(\sup_{n\geq 0}\nu_n<\infty\right)=1.\end{equation}
\finA

We introduce a new pseudo-metric on $\Theta$. For $\beta\in[0,1]$, $\theta,\theta'\in\Theta$, set
\[
D_\beta(\theta,\theta') \eqdef \sup_{|f|_{V^\beta}\leq 1}\;\sup_{x \in \Xset} \frac{\left|P_{\theta}f(x) - P_{\theta'}f(x) \right|}{V^\beta(x)} \eqsp.
\]

Under A\ref{A1} and A\ref{A2} a weak law of large numbers hold.
\begin{theo}\label{thm1}
Assume A\ref{A1}-A\ref{A2}. Let $\beta\in [0,1-\alpha)$ and $f_\theta:\Xset \to\rset$ a family of measurable functions of $\L_{V^\beta}$ such that $\pi(f_\theta)=0$, $\theta\to f_\theta(x)$ is measurable and $\sup_{\theta\in\compact}|f_\theta|_{V^\beta}<\infty$ for any compact subset $\compact$ of $\Theta$. Suppose also that there exist $\epsilon>0$, $\kappa>0$, $\beta+\alpha\kappa<1-\alpha$ such that for any $(x,\theta,l)\in\Xset_0\times\Theta_0\times \nset$
\begin{equation}\label{diminish}
\PE_{x,\theta}^{(l)}\left[\sum_{k\geq 1}k^{-1+\epsilon}\left(D_\beta(\tilde \theta_k,\tilde \theta_{k-1})
+|f_{\tilde \theta_k}-f_{\tilde \theta_{k-1}}|_{V^\beta}\right)\un_{\{\tauout_{\compact_l}>k\}}V^{\beta+\alpha\kappa}(\tilde X_k)\right]<\infty.\end{equation}
Then $n^{-1}\sum_{k=1}^nf_{\theta_{k-1}}(X_k)$ converges in $\cPP$-probability to zero.
\end{theo}
\begin{proof}The proof is given in Section \ref{sec:proofthm1}.\end{proof}

\begin{rem}
A strong law of large numbers also hold under similar assumptions (\cite{atchadeetfort08}). It is an open problem whether A\ref{A1}, A\ref{A2} and (\ref{diminish}) imply a weak law of large numbers hold for measurable functions $f$ for which $\pi(|f|)<\infty$ without the additional assumption that $f\in\L_{V^\beta}$, $0\leq \beta<1-\alpha$.
\end{rem}

For the Central limit theorem, we introduce few additional notations. For $f\in\L_{V^\beta}$ with $\pi(f)=0$, and $a\in [0,1/2]$ we introduce the resolvent functions
\[g_a(x,\theta)=\sum_{j\geq 0}(1-a)^{j+1} P_{\theta}^j f(x).\]
Whenever $g_a$ is well defined it satisfies the \textit{approximate Poisson equation} \begin{equation}\label{approxintro}f(x)=(1-a)^{-1}g_a(x,\theta)-P_\theta g_a(x,\theta).\end{equation}
When $a=0$, we write $g(x,\theta)$ which is the usual solution to the Poisson equation $f(x)=g(x,\theta)-P_\theta g(x,\theta)$. Define also
\begin{equation}\label{HafunIntro}
H_a(x,y)=g_a(y,\theta)-P_\theta g_a(x,\theta),\end{equation}
where $P_\theta g_a(x,\theta)\eqdef \int P_\theta(x,dz)g_a(z,\theta)$. We start by showing that under A\ref{A1}-A\ref{A2}, the partial sum $\sum_{k=1}^nf(X_k)$ admits a martingale approximation.

\begin{theo}\label{thm2}Assume A\ref{A1}-A\ref{A2} with $\alpha<1/2$. Let $\beta\in[0,\frac{1}{2}-\alpha)$ and $f\in\L_{V^\beta}$ such that $\pi(f)=0$. Let $\kappa>1$, $\delta\in (0,1)$ be such that  $2\beta+\alpha(\kappa+\delta)<1-\alpha$. Take $\rho\in(\frac{1}{2},\frac{1}{2-\delta}]$ and let $\{a_n,\;n\geq 0\}$ be any sequence of positive numbers such that $a_n\in (0,1/2]$, $a_n\propto n^{-\rho}$. Suppose that for any $(x,\theta,b,l)\in \Xset_0\times \Theta_0\times [0,1-\alpha]\times\nset$
\begin{equation}\label{diminishclt}
\PE^{(l)}_{x,\theta}\left[\sum_{k\geq 1}\un_{\{\tauout_\compact>k\}}k^{-1+\rho(2-\delta)}D_{b}(\tilde \theta_k,\tilde \theta_{k-1})
V^{2\beta+\alpha(\kappa+\delta)}(\tilde X_k)\right]<\infty.\end{equation}
Then
\begin{equation*}\lim_{n\to\infty}n^{-1/2}\sum_{k=1}^n\left(f(X_k)-H_{a_{n},\theta_{k-1}}(X_{k-1},X_k)\right)=0,\;\;\;\mbox{ in }\;\cPP\mbox{-probability}.\end{equation*}
\end{theo}
\begin{proof}We show in Lemma \ref{prop:llnhalf} that the same martingale approximation hold for the re-projection free process $\{(\tilde X_n,\tilde\theta_n),\;n\geq 0\}$ and this property transfers to the adaptive chain $\{(X_n,\theta_n,\nu_n,\xi_n),\;n\geq 0\}$ as a consequence of Lemma \ref{prop:conv}.\end{proof}

The process $\sum_{j=1}^k n^{-1/2}H_{a_{n},\theta_{j-1}}(X_{j-1},X_j)\un_{\{\xi_{j-1}>0\}}$, $1\leq k\leq n$ is a martingale array but do not satisfy a CLT in general. To derive a CLT we strengthen A\ref{A2}.

\debutA
\item \label{A3} There exists a $\Theta$-valued random variable $\theta_\star$ such that with $\cPP$-probability one, $\{\theta_n,\;n\geq 0\}$ remains in a compact set and $\lim_{n\to\infty}D_\beta(\theta_n,\theta_\star)=0$ for any $\beta\in[0,1-\alpha]$.
\finA

Notice that the compact set referred to in A\ref{A3} is sample path dependent.

\begin{theo}\label{thm3}
Assume A\ref{A1} and A\ref{A3} with $\alpha<1/2$. Let $\beta\in [0,\frac{1}{2}-\alpha)$, $f\in\L_{V^\beta}$, $\kappa,\delta,\rho$ and $\{a_n,\;n\geq 0\}$ as in Theorem \ref{thm2}. Suppose that the diminishing adaptation condition (\ref{diminishclt}) hold and
\begin{equation}\label{A4}\lim_{n\to\infty}\frac{1}{n}\sum_{k=1}^ng_{a_n}^2(X_k,\theta_{k-1})-P_{\theta_{k-1}}g_{a_n}^2(X_{k},\theta_{k-1})=0,\;\;\;\mbox{ in } \;\;\cPP\mbox{-probability}.\end{equation}
Then there exists a nonnegative random variable $\sigma^2_\star(f)$ such that $n^{-1/2}\sum_{k=1}^nf(X_k)$ converges weakly to a random variable $Z$ with characteristic function $\phi(t)=\cPE\left[\exp\left(-\frac{\sigma_\star^2(f)}{2}t^2\right)\right]$. Moreover
\[\sigma_\star^2(f)=\int\pi(dx)\left\{2f(x)g(x,\theta_\star)-f^2(x)\right\},\;\;\;\cPP-a.s.\]
\end{theo}
\begin{proof}See Section \ref{sec:proofthm3}.\end{proof}

\subsection{On assumption (\ref{A4})}
Assumption (\ref{A4}) is needed to establish the weak law of large numbers in the CLT. When $\{X_n,\;n\geq 0\}$ is a stationary  Markov chain (\ref{A4}) automatically hold. The proof is based on a result due to \cite{mw00}. The stationarity assumption is not restrictive in the case of Harris recurrent Markov chain.

\begin{prop}\label{ExMC}
Suppose that $\{X_n,\;n\geq 0\}$ is a stationary and ergodic Markov chain with invariant distribution $\pi$ and transition kernel $P$ that satisfies (\ref{drift}) and (\ref{rateconv}) with $\alpha<1/2$. Let $f\in \L_{V^\beta}$ with $\beta\in[0,1/2-\alpha)$. Then (\ref{A4}) hold.
\end{prop}
\begin{proof}See Section \ref{sec:proofpropMC}. \end{proof}

In the general adaptive case, the simplest approach to checking (\ref{A4}) is through appropriate moments condition.
\begin{prop}\label{propMomCLT}
Assume A\ref{A1} and A\ref{A3} with $\alpha<1/2$. Let $\beta\in [0,\frac{1}{2}-\alpha)$, $f\in\L_{V^\beta}$, $\kappa,\delta,\rho$ and $\{a_n,\;n\geq 0\}$ as in Theorem \ref{thm2}. Suppose that there exists $\epsilon>0$ such that for any $(x,\theta,l)\in\Xset_0\times \Theta_0\times \nset$
\begin{equation}\label{eq:MomCondCLT}\sup_{n\geq 1}\;n^{-1}\PE^{(l)}_{x,\theta}\left[\sum_{k=1}^nV^{2(\beta+\alpha)+\epsilon}(\tilde X_k)\right]<\infty.\end{equation}
Then (\ref{A4}) hold.
\end{prop}
\begin{proof}See Section \ref{sec:proofpropMomCLT}.\end{proof}

One can always check (\ref{eq:MomCondCLT}) if $\alpha<1/3$ and $\beta\in [0,1-3\alpha)$.

\begin{coro}
Assume A\ref{A1} and A\ref{A3} with $\alpha<1/3$. Let $\beta\in [0,1-3\alpha)$, $f\in\L_{V^\beta}$, $\kappa,\delta,\rho$ and $\{a_n,\;n\geq 0\}$ as in Theorem \ref{thm2}. Suppose that (\ref{diminishclt}). Then the conclusion of Theorem \ref{thm3} hold.
\end{coro}\label{corthm3}
\begin{proof} If $\alpha<1/3$ and we take $\beta\in [0,1-3\alpha)$ then we can find $\epsilon>0$ such that $2(\beta+\alpha)+\epsilon<1-\alpha$ and by Proposition \ref{prop:useful} (ii), Eq. (\ref{eq:MomCondCLT}) hold. The stated result thus follows from Proposition \ref{propMomCLT} and Theorem \ref{thm3}.\end{proof}

\subsection{Some additional remarks on the assumptions}
\subsubsection{On Assumption A\ref{A1}}
In many cases, A\ref{A1} can be checked by establishing a drift and a minorization conditions. For example if uniformly over compact subsets $\compact$ of $\Theta$, $P_\theta$ satisfies  a polynomial drift condition of the form $P_\theta V\leq V -cV^{1-\alpha} +b \un_{\Cset}$ for some small set $\Cset$, $\alpha\in (0,1]$ and such that the level sets of $V$ are $1$-small then (\ref{drift}) and (\ref{rateconv}) hold. This point is thoroughly discussed in \cite{atchadeetfort08} (Section 2.4 and Appendix A) and the references therein.

Assumption A\ref{A1} also hold for geometrically ergodic Markov kernels and in this case we recover the CLT result of \cite{andrieuetal06}. Indeed, suppose that uniformly over compact subsets $\compact$ of $\Theta$, there exist $\Cset\in\Xsigma$, $\nu$ a probability measure on $(\Xset,\Xsigma)$, $b,\epsilon>0$ and $\lambda\in (0,1)$ such that $\nu(\Cset)>0$, $P_\theta(x,\cdot)\geq \epsilon\nu(\cdot)\un_\Cset(x)$ and $P_\theta V\leq \lambda V+b\un_\Cset$. Then for any $\alpha\in (0,1]$, $P_\theta V\leq V-(1-\lambda) V^{1-\alpha}+b$, thus (\ref{drift}) hold. Moreover by explicit convergence bounds for geometrically ergodic Markov chains (see e.g. \cite{baxendale05}), for any $\beta\in (0,1]$
\[\sup_{\theta\in\compact}\|P^n_\theta(x,\cdot)-\pi(\cdot)\|_{V^\beta}\leq C_\beta(\compact)\rho_\beta^n V^\beta(x).\]
A fortiori (\ref{rateconv}) hold. Also under the geometric drift condition, if $\beta\in [0,1/2)$ then we can find $0<\alpha<1/2$ and $\epsilon>0$ such that $2(\beta+\alpha)+\epsilon<1$, and since $V^\delta$-moment of geometrically ergodic adaptive MCMC are bounded in $n$ for any $\delta\in [0,1)$, we get (\ref{eq:MomCondCLT}). In this case and assuming (\ref{diminishclt}), Theorem \ref{thm3} yields a CLT for all functions $f\in\L_{V^\beta}$ with $\beta\in [0,1/2)$ which is the same CLT obtained in \cite{andrieuetal06} (Theorem 8). Roughly speaking, assuming (\ref{diminishclt}) at no extra cost is similar to setting $\beta=0$ in their theorem).

\subsubsection{On assumption A\ref{A2}-A\ref{A3}}
Assumption A\ref{A3} is a natural assumption to make when a CLT is sought. Whether A\ref{A2} or A\ref{A3} hold depends on the adaptation strategies. We show below how to check A\ref{A3} when the adaptation is driven by stochastic approximation.


\subsubsection{On the diminishing adaptation conditions (\ref{diminish}) and (\ref{diminishclt})}
It is well known that adaptive MCMC can fail to converge when to so-called \textit{diminishing adaptation} condition (which embodies the idea that one should adapt less and less with the iterations) does not hold. Here, the diminishing adaptation takes the form of conditions (\ref{diminish}) and (\ref{diminishclt}). Indeed, (\ref{diminish}) and (\ref{diminishclt}) cannot hold unless $D_\beta(\theta_n,\theta_{n-1})$ converges to zero in some sense. These conditions are not difficult to check. Typically $D_b(\theta_k,\theta_{k-1})\leq C\gamma_kV^\eta(X_k)$ for some positive numbers $\gamma_k$ and $\eta\geq 0$. then we can check (\ref{diminish}) or (\ref{diminishclt}) using Proposition \ref{lem:sumdrift}.

\subsection{Checking A\ref{A3} for AMCMC driven by stochastic approximation}\label{sec:SA}
Adaptive MCMC is often driven by stochastic approximation. We consider an example of stochastic approximation dynamics and show how to check A\ref{A3}. Let $\{\gamma_n\}$ be a sequence of positive numbers. Let $q^{(1)}_\theta:\;\Xset\times \Xsigma\to [0,1]$ and $q^{(2)}_\theta:\Xset\times\Xset\times \Xsigma\to [0,1]$ be two Markov kernels such that
\[\int q^{(1)}_\theta(x,dy)q_\theta^{(2)}\left((x,y),dx'\right)=P_{\theta}(x,dx').\]
Let $\Phi:\;\Theta\times\Xset\times\Xset\to\Theta$ be a measurable function. For convenience we write $\Phi_\theta(x,y)$ instead of $\Phi(\theta,x,y)$. We consider the adaptive MCMC algorithm with the kernels $\bar R$ are given as
\begin{equation}\label{eq:dynSA}
\bar R\left(n;(x,\theta),(dx',d\theta')\right)=\int q^{(1)}_\theta(x,dy)q_\theta^{(2)}\left((x,y),dx'\right)\delta_{\theta+\gamma_n\Phi_\theta(x,y)}(d\theta').\end{equation}

Under (\ref{eq:dynSA}), the dynamics on $\theta_n$ in algorithm \ref{algo1} can then be written as
\begin{equation*}
\theta_{n+1}=\theta_n+\gamma_{\nu_n+\xi_n}\left(h(\theta_n)+\epsilon^{(1)}_{n+1}+\epsilon^{(2)}_{n+1}\right),\;\;\; \mbox{ on }\{\xi_n>0\},\;\;\;\;\cPP-a.s.\end{equation*}
where $\epsilon_{n+1}^{(1)}=\Upsi_{\theta_n}(X_n)-h(\theta_n)$, $\epsilon_{n+1}^{(2)}=\Phi_{\theta_n}(X_n,Y_{n+1})-\Upsi_{\theta_n}(X_n)$, where $Y_{n+1}$ is a random variable with conditional distribution $q_{\theta_n}^{(1)}(X_n,\cdot)$ given $\check\F_n$  and where
\[\Upsi_\theta(x)=\int q^{(1)}_{\theta}(x,dy)\Phi_{\theta}(x,y),\;\;\;\;\mbox{ and }\;\;\;h(\theta)=\int\pi(dx)\Upsi_\theta(x).\]

Following \cite{AMP05}, we assume that

\debutB
\item \label{B1}
\begin{enumerate}
\item $\{\compact_n, n\geq 0\}$ is such that $\Theta = \bigcup_n \compact_n$, $\Theta_0\subseteq\compact_0$  and $\compact_n \subset \mbox{int}(\compact_{n+1})$, where $\mbox{int}(A)$ is the interior of $A$.

\item The function $h$ is a continuous function and there exists a continuously differentiable function $w:\;\Theta\to [0,\infty)$ such that
\begin{enumerate}
\item  for any $\theta\in\Theta$, $\seq{\nabla w(\theta),h(\theta)}\leq 0$, the set $\L\eqdef\{\theta\in\Theta:\; \seq{\nabla w(\theta),h(\theta)}=0\}$ is non-empty and the closure of $w(\L)$ has an empty interior.
\item  there exists $M_0>0$ such that $\L\cup\Theta_0\subset\{\theta:\;w(\theta)<M_0\}$ and for any $M\geq M_0$, $\W_M\eqdef \{\theta:\; w(\theta)\leq M\}$ is a compact set.
\end{enumerate}
\end{enumerate}
\finB

\vspace{0.3cm}

For integers $p\geq 0$, $n\geq 1$ and a compact subset $\compact$ of $\Theta$, we define  the random variable
\[C_{n,p}(\compact)\eqdef \sup_{l\geq n}\un_{\{\tauout_{\compact}>l\}}\left|\sum_{j=n}^l\gamma_{p+j-1}\left(\tilde\epsilon^{(1)}_{j}+\tilde \epsilon^{(2)}_{j}\right)\right|,\]
where $\tilde \epsilon^{(1)}_{n+1}=\Upsi_{\tilde \theta_{n}}(\tilde X_{n})-h(\tilde \theta_{n})$ and $\tilde\epsilon^{(2)}_{n+1}=\Phi_{\tilde\theta_n}\left(\tilde X_n,\tilde Y_{n+1}\right)-\int q_{\tilde\theta_n}^{(1)}(\tilde X_n,dy)\Phi_{\tilde\theta_n}\left(\tilde X_n,y\right)$ and where the conditional distribution of $\tilde Y_{n+1}$ given $\F_n$ is $q_{\tilde\theta_n}^{(1)}(\tilde X_n,\cdot)$.

$C_{n,p}(\compact)$ is the magnitude of the errors in the stochastic approximation. Notice that $C_{n,p}(\compact)$ is defined from the re-projection free process. A key  result shown by \cite{AMP05} is that when B\ref{B1} hold, the convergence of a SA algorithm depends mainly on $C_{n,p}(\compact)$. The framework considered here is slightly different from \cite{AMP05} but the result still hold. The proof follows the same lines as in \cite{AMP05} and we omit the details.

\begin{prop}\label{propcvSA}
Assume (\ref{eq:dynSA}), B\ref{B1}, $\lim_n\gamma_n=0$ and $\sum_n\gamma_n=\infty$. Suppose that for any $M>0$ large enough and for any $\delta>0$
\begin{equation}\label{eqcvSA1}
\lim_{p\to\infty}\sup_{(x,\theta)\in\Xset_0\times\Theta_0}\PP_{x,\theta}^{(p)}\left(C_{1,p}(\W_M)>\delta\right)=0,\end{equation}
and for any $p\geq 0$,
\begin{equation}\label{eqcvSA2}
\lim_{n\to\infty}\sup_{(x,\theta)\in\Xset_0\times\Theta_0}\PP_{x,\theta}^{(p)}\left(C_{n,p}(\compact_p)>\delta\right)=0.\end{equation}
Then A\ref{A3} hold.
\end{prop}

We now show that (\ref{eqcvSA1})-(\ref{eqcvSA2}) hold true under A\ref{A1}.

Assume that the function $\Upsi$ satisfies
\debutB
\item \label{B2}
There exists $\eta\geq 0$, $2(\eta+\alpha)<1$ such that for any compact subset $\compact$ of $\Theta$, $b\in [0,1-\alpha]$, $\theta,\theta'\in\compact$,
\begin{equation}
\sup_{\theta\in\compact}\sup_{x\in\Xset}V^{-2\eta}(x)\int q_\theta^{(1)}(x,y)\left|\Phi_\theta(x,y)\right|^2 <\infty,\;\;\mbox{ and }\;\; D_b(\theta,\theta') + |\Upsi_\theta-\Upsi_{\theta'}|_{V^\eta}\leq C|\theta-\theta'|,\end{equation}
for some finite constant $C$ that depends possibly on $\compact$.
\finB

\begin{prop}\label{propcvSA2}
Assume A\ref{A1} with $\alpha<1/2$ and (\ref{eq:dynSA}). Suppose that B\ref{B1} and B\ref{B2} hold. Suppose also that $\lim_n\gamma_n=0$ and $\sum_n\gamma_n=\infty$ and for any $p\geq 0$,
\begin{equation}\label{eq1thm3}
\lim_{n\to\infty}(\gamma_{p+n-1}-\gamma_{p+n})n^{1-\alpha}=0\;\; \mbox{ and }\;\; \sum_{n\geq 1}\left(\gamma_k^2 k^\rho \; + \gamma_k k^{-\rho} \; + \gamma_k^{1+\rho}\right)<\infty,\end{equation}
for some $\rho\in \left(0,(1-\alpha)(\eta+\alpha)^{-1}-1\right)$. Then A\ref{A3} hold.
\end{prop}
\begin{proof}See Section \ref{proofpropcvSA2}.\end{proof}

\subsection{Example: Adaptive Langevin algorithms}\label{sec:ex}
We illustrate the theory above with an application to the Metropolis-adjusted Langevin algorithm (MALA). In this section, $\Xset$ is the $d$-dimensional Euclidean space $\rset^d$ and $\pi$ is a positive density on $\Xset$ with respect to the Lebesgue (denoted $\mu_{Leb}$ or $dx$). The MALA algorithm is an effective Metropolis-Hastings algorithm whose proposal kernel is obtained by discretization of the Langevin diffusion
\[dX_t=\frac{1}{2}e^{\theta}\nabla\log\pi(X_t) dt +e^{\theta} dB_t,\;\;\;X_0=x,\]
where $\theta\in\rset$ is a scale parameter and $\{B_t,\;t\geq 0\}$ a $d$-dimensional standard Brownian motion. Denote $q_{\theta}(x,y)$ the density of the $d$-dimensional Gaussian distribution with mean $b_{\theta}(x)$ and covariance matrix $e^{\theta} I_d$ where
\[b_{\theta}(x)=x+\frac{1}{2}e^{\theta}\nabla\log \pi(x).\]
The MALA works as follows. Given $X_n=x$, we propose a new value $Y\sim q_{\theta}(x,\cdot)$. Then with probability $\alpha_{\theta}(X_n,Y)$, we 'accept $Y$' and set $X_{n+1}=Y$ and with probability $1-\alpha_{\theta}(X_n,Y)$, we 'reject $Y$' and set $X_{n+1}=X_n$. The acceptance probability is given by
\[\alpha_{\theta}(x,y)=1\wedge \frac{\pi(y)q_{\theta}(y,x)}{\pi(x)q_{\theta}(x,y)}.\]

The convergence and optimal scaling of MALA is studied in detail in \cite{robertsettweedie96b,robertsetrosenthal01}. In practice the performance of this algorithm depends on the choice of the scale parameter $\theta$. In high-dimensional spaces (and under some regularity conditions) it is optimal to set $\theta=\theta_\star$ such that the average acceptance probability of the algorithm in stationarity is $0.574$. In general, $\theta_\star$ is not available and its computation would require a tedious fine-tuning of the sampler.  Adaptive MCMC provides a straightforward approach to properly scale the algorithm.

The parameter space is $\Theta=\rset$. For $\theta\in\Theta$, denote $P_\theta$ the transition kernel of the
MALA algorithm with proposal $q_{\theta}$. We also introduce the functions
\[A_{\theta}(x)\eqdef\int_\Xset \alpha_{\theta}(x,y)q_{\theta}(x,y)\mu_{Leb}(dy),\;\;\; a(\theta)\eqdef\int_\Xset A_{\theta}(x)\pi(x)\mu_{Leb}(dx).\]

Let $\{\compact_n,\;n\geq 0\}$ be a family of nonempty compact intervals of $\Theta$ such that $\cup\compact_n=\rset$, $\compact_n\subset\mbox{int}(\compact_{n+1})$. Therefore by construction B\ref{B1}-(1) hold. Let $\Theta_0=\{\theta_0\}$ and $\Xset_0=\{x_0\}$ for some arbitrary point $(x_0,\theta_0)\in\Xset\times\compact_0$. The re-projection function is  $\Pi(x,\theta)=(x_0,\theta_0)$ for any $(x,\theta)\in\Xset\times\Theta$. We also have $\Pi_k(x,\theta)=(x,\theta)$ if $k>0$ and $\Pi_k(x,\theta)=\Pi(x,\theta)$ if $k=0$. Obviously many other choices are possible. The adaptive MALA we consider is the following.
\begin{algo}\label{algo2}
\begin{description}
\item [Initialization] Let $\bar\alpha$ be the target acceptance probability (taken as $0.574$). Choose $(X_0,\theta_0)\in \Xset_0\times \Theta_0$, $\nu_0=0$ and $\xi_0=0$.
\item [Iteration] Given $(X_n,\theta_n,\nu_n,\xi_n)$: set $(\bar X,\bar\theta)=\Pi_{\xi_n}(X_n,\theta_n)$.
\begin{description}
\item [a] generate $Y_{n+1}\sim q_{\bar \theta}\left(\bar X,\cdot\right)$. With probability $\alpha_{\bar\theta}(\bar X,Y_{n+1})$, set $X_{n+1}=Y_{n+1}$ and with probability $1-\alpha_{\bar\theta}(\bar X,Y_{n+1})$, set $X_{n+1}=\bar X$.
\item [b] Compute
\begin{equation}\label{recurs}
\theta_{n+1}=\bar \theta+\frac{1}{1+\nu_n+\xi_n}\left(\alpha_{\bar\theta}(\bar X,Y_{n+1})-\bar\alpha\right).\end{equation}
\item [c] If $\theta_{n+1}\in \compact_{\nu_n}$ then set $\nu_{n+1}=\nu_n$ and $\xi_{n+1}=\xi_n+1$. Otherwise if $\theta_{n+1}\notin \compact_{\nu_n}$ then set $\nu_{n+1}=\nu_n+1$ and $\xi_{n+1}=0$.
\end{description}
\end{description}
\end{algo}

In this algorithm, the kernel $\bar R(n;\cdot,\cdot)$ takes the form
\begin{multline*}
\bar R\left(n;(x,\theta),(dx',d\theta')\right)=\int q_{\theta}(x,dy)\left(\alpha_{\theta}(x,y)\delta_y(dx')\right.\\
\left.+(1-\alpha_{\theta}(x,y))\delta_x(dx')\right)\delta_{\Phi_n(\theta,x,y)}(d\theta'),\end{multline*}
where $\Phi_n(\theta,x,y)=\theta+(n+1)^{-1}(\alpha_{\theta}(x,y)-\bar\alpha)$. Thus (\ref{eq:dynSA}) hold. We make the following assumption.

\debutC
\item  \label{C1} $\bar\alpha\in (0,1)$, $\lim_{\theta\to+\infty}a(\theta)=0$, $\lim_{\theta\to -\infty}a(\theta)=1$.
\finC

\begin{prop}\label{proplyapEx}
Under C\ref{C1}, the function $h(\theta)=a(\theta)-\bar\alpha$ satisfies B\ref{B1}-(2) with $\L=\{\theta\in\rset:\; a(\theta)=\bar\alpha\}$ and $w(\theta)=\int_0^\theta\cosh(u)(\bar\alpha-a(u))du + K$ for some finite constant $K$ where $\cosh(u)=(e^u+e^{-u})/2$ is the hyperbolic cosine.
\end{prop}
\begin{proof}See Section \ref{proofproplyapEx}.\end{proof}

We assume that the target density $\pi$ is heavy tailed as in \cite{kamatani08}.

\debutC
\item  \label{C2} We assume that $\pi:\rset^d\to (0,\infty)$ is of class $\mathcal{C}^2$ and there exists $\eta>d$ such that
\begin{equation}
\limsup_{|x|\to\infty}\seq{x,\nabla\log\pi(x)}\leq -\eta,\;\;\lim_{|x|\to\infty}\left|\nabla\log\pi(x)\right|=0,\;\; \lim_{|x|\to\infty}\|\nabla^2\log\pi(x)\|=0,\end{equation}
where for a matrix $A$, $\|A\|$ denotes its Frobenius norm.
\finC

The next proposition is a paraphrase of Theorem 5 of \cite{kamatani08}.
\begin{prop}\label{propdriftMALA}
Assume C\ref{C2}. For $s\in (2,2+\eta-d)$, define $V_s(x)=\left(1+|x|^2\right)^{s/2}$ and $\alpha=2/s$. Let $\Cset$ be a compact subset of $\rset^d$ with $\mu_{Leb}(\Cset)>0$. For any compact subset $\compact$ of $\Theta$, there exists $\epsilon,c,b\in (0,\infty)$, such that
\[\inf_{\theta\in\compact} P_\theta(x,dy)\geq \epsilon\left[\frac{\mu_{Leb}(dy)\un_\Cset(y)}{\mu_{Leb}(\Cset)}\right]\un_\Cset(x),\]
\[\sup_{\theta\in\compact} P_\theta V_s(x)\leq V_s(x)-c V^{1-\alpha}(x)+b\un_\Cset(x).\]
\end{prop}

For the smoothness we have
\begin{prop}\label{propstabEx}
Assume that $\left|\nabla\log\pi(x)\right|$ is a bounded function. Let $\compact$ be a compact convex subset of $\Theta$. There exists a finite constant $C(\compact)$ such that for any $f\in\L_{V_s^{\beta}}$, $\beta\in[0,1]$, any $\theta,\theta'\in\compact$,
\begin{equation}
\left|\int\alpha_\theta(x,y)q_{\theta}(x,y)f(y)dy-\int\alpha_{\theta'}(x,y)q_{\theta'}(x,y)f(y)dy\right|\leq C(\compact)|f|_{V_s^{\beta}} |\theta-\theta'|V_s^{\beta}(x).\end{equation}
\end{prop}
\begin{proof}
See Section \ref{proofpropstabEx}.
\end{proof}

We now apply Theorem \ref{thm3} to get a CLT for the adaptive MALA.
\begin{theo}\label{thm4}
Assume C\ref{C1} and C\ref{C2} with $\eta>d+4$. Let $s\in(6,2+\eta-d)$ and let $f:\;\Xset\to\rset$ be a measurable function such that $\pi(f)=0$ and $|f(x)|\leq C(1+|x|^2)^b$ for some $b\in [0,\frac{s}{2}-3)$ and some finite constant $C$. Then there exists a nonnegative random variable $\sigma^2_\star(f)$ such that $n^{-1/2}\sum_{k=1}^n f(X_k)$ converges weakly to a random variable $Z$ with characteristic function $\phi(t)=\cPE\left[\exp\left(-\frac{\sigma_\star^2(f)}{2}t^2\right)\right]$.
\end{theo}
\begin{rem}
If $\pi$ is positive and of class $\mathcal{C}^2$ and $\pi(x)\approx (1+|x|^2)^{-(d+\nu)/2}$ in the tails, then $C\ref{C2}$ hold with $\eta=\nu+d$ and Theorem \ref{thm4} guarantees a CLT for $\nu>4$. Compare with $\nu>2$ for Harris recurrent Markov chains satisfying A\ref{A1}.\end{rem}
\begin{proof}
A\ref{A1} hold as a consequence of Proposition \ref{propdriftMALA} (see e.g. \cite{atchadeetfort08}~Section 2.4 and Appendix A). Proposition \ref{proplyapEx} shows that B\ref{B1}-(2) hold and Proposition \ref{propstabEx} implies that B\ref{B2} hold. Therefore A\ref{A3} hold as a consequence of Proposition \ref{propcvSA2}. (\ref{diminishclt}) is an easy consequence of Proposition \ref{propstabEx} and Proposition \ref{lem:sumdrift}. We thus conclude with Corollary \ref{corthm3}.
\end{proof}

In the above theorem the asymptotic variance $\sigma_\star^2(f)$ takes values in the set $\{\sigma^2_\theta(f),\; \theta\in\L\}$, where $\L=\{\theta\in\rset:\; a(\theta)=\bar\alpha\}$ and
\[\sigma^2_\theta(f)\eqdef \int\pi(dx)\left\{f^2(x)+2\sum_{k\geq 0} f(x)P_\theta^k f(x)\right\}.\]
In particular, if $\L=\{\theta_\star\}$ and $\sigma_{\theta_\star}^2(f)>0$, then $n^{-1/2}\sum_{k=1}^n f(X_k)$ converges weakly to $\mathcal{N}\left(0,\sigma_{\theta_\star}^2(f)\right)$.


\section{Proofs}\label{sec:proofs}
The proofs are organized as follows. The weak law of large numbers (Theorem \ref{thm1}) is proved in Section \ref{sec:proofthm1}, the CLT (Theorem \ref{thm3}) is proved in Section \ref{sec:proofthm3}. In Section \ref{sec:approxpoisson} we develop some preliminary results on the resolvent functions $g_a$ and we establish some basic results on the asymptotic behavior of the nonhomogeneous process $\{(\tilde X_n,\tilde \theta_n),\;n\geq 0\}$  in Section \ref{sec:moments}-\ref{sec:llnwrep}. The results in Section \ref{sec:link} (in particular Lemma \ref{prop:conv}) serve as a link and allow us to reduce the limiting behavior of the adaptive algorithm $\{(X_n,\theta_n,\nu_n,\xi_n),\;n\geq 0\}$ to that of the nonhomogeneous Markov chain $\{(\tilde X_n,\tilde \theta_n),\;n\geq 0\}$.

Throughout the proof, $C(\compact)$ denotes a finite constant that depends on the compact set $\compact$ and on the constants in the above assumptions. But to simplify the notations, we will not keep track of these constants so the actual value of $C(\compact)$ might be different from one appearance to the next.

\subsection{Resolvent kernels and approximate Poisson's equations}\label{sec:approxpoisson}
In this section, $\compact$ is a given compact subset of $\Theta$ and $\beta\in [0,1-\alpha]$. We consider a family of functions $f_\theta\in \L_{V^\beta}$, $\theta\in\Theta$ such that $\pi(f_\theta)=0$. For $a\in(0,1)$ we define the resolvent function associated with $f_\theta$ as
\[\tilde g_a(x,\theta)=\sum_{j=0}^\infty (1-a)^{j+1}P_\theta^jf_\theta(x)=\sum_{j=0}^\infty (1-a)^{j+1}\bar P_\theta^jf_\theta(x),\]
where $\bar P_\theta=P_\theta-\pi$. Similarly we define
\[\tilde g(x,\theta)=\sum_{j=0}^\infty P_\theta^jf_\theta(x)=\sum_{j=0}^\infty\bar P_\theta^jf_\theta(x),\]
When $f_\theta\equiv f$ does not depend on $\theta\in\Theta$, and to help keep the notation clear, we write $g_a(x,\theta)$ (resp. $g(x,\theta)$) instead of $\tilde g_a(x,\theta)$ (resp. $\tilde g$). It is easy to see that when $\tilde g_a$ is well defined, it satisfies the following approximate Poisson equation
\begin{equation}\label{approxpoisson}
f_\theta(x)=(1-a)^{-1}\tilde g_a(x,\theta)-P_\theta \tilde g_a(x,\theta).\end{equation}
Similarly $\tilde g$, when well-defined,  satisfies the Poisson equation
\begin{equation}\label{poisson}
f_\theta(x)=\tilde g(x,\theta)-P_\theta \tilde g(x,\theta).\end{equation}
We introduce the function
\begin{equation*}\zeta_\kappa(a)\eqdef \sum_{j\geq 0}(1-a)^{j+1}(1+j)^{-\kappa}.\end{equation*}
We will need the following lemma.
\begin{lemma}\label{lemseq}
For any $a\in(0,1/2]$ and $\kappa\geq 0$,
\[\zeta_\kappa(a)\leq \left\{\begin{array}{ll}\sum_{k\geq 1}k^{-\kappa}\;& \mbox{ if }\; \kappa>1\\
-\log(2a) +1\; & \mbox{ if }\kappa=1\\ 2^{-1+\kappa}\Gamma(1-\kappa)a^{-1+\kappa}\; & \mbox{ if }\; 0\leq \kappa<1\end{array} \right.,\]
where $\Gamma(x):=\int_0^\infty u^{x-1}e^{-u}du$ is the Gamma function.
\end{lemma}
\begin{proof}
$(1-a)^j\leq 1$ for all $j\geq 1$. Therefore, for $\kappa>1$, $\sum_{j\geq 0}(1-a)^{j+1}(1+j)^{-\kappa}\leq \sum_{j\geq 1}j^{-\kappa}$. For $\kappa=1$, we note that $\frac{d}{d\;a}\left\{\sum_{j\geq 0}(1-a)^{j+1}(1+j)^{-\kappa}\right\}=-a^{-1}$. Therefore for $a\in (0,1/2]$, $\sum_{j\geq 0}(1-a)^{j+1}(1+j)^{-\kappa}=\sum_{j\geq 1}(j2^j)^{-1}-\log(2a)\leq -\log(2a) +1$. Finally, if $0\leq \kappa<1$, by monotonicity, $\sum_{j\geq 0}(1-a)^{j+1}(1+j)^{-\kappa}\leq \int_0^\infty (1-a)^xx^{-\kappa}dx=\int_0^\infty x^{1-\kappa-1}e^{-\beta x}dx=\Gamma(1-\kappa)\beta^{-1+\kappa}$, where $\beta=-\log(1-a)$. For $a\in(0,1/2]$, $-\log(1-a)\leq 2a$ and we conclude that $\sum_{j\geq 0}(1-a)^{j+1}(1+j)^{-\kappa}\leq 2^{-1+\kappa}\Gamma(1-\kappa)a^{-1+\kappa}$.
\end{proof}

\begin{prop}\label{propboundga}Assume A\ref{A1}.
\begin{description}
\item [(i)]  Let $\kappa\in[0,\alpha^{-1}(1-\beta)-1]$. There exists a finite constant $C(\compact)$ such that for any  $(x,\theta)\in\Xset\times \compact$ and any $a\in (0,1/2]$
\begin{equation}\label{boundga}
\left|\tilde g_a(x,\theta)\right|\leq C(\compact)|f_\theta|_{V^\beta}\zeta_\kappa(a)V^{\beta+\alpha\kappa}(x).\end{equation}
\item [(ii)]Suppose that $\alpha<1/2$. Let $\kappa\in(1,\alpha^{-1}(1-\beta)-1]$. There exists a finite constant $C(\compact)$ such that for any  $(x,\theta)\in\Xset\times \compact$ and any $a\in (0,1/2]$
\begin{equation}\label{boundgag}\left|\tilde g_a(x,\theta)-\tilde g(x,\theta)\right|\leq C(\compact)|f_\theta|_{V^\beta}\left(2^{1-\kappa}\int_0^a\zeta_{\kappa-1}(u)du\right)V^{\beta+\alpha\kappa}(x).\end{equation}
\end{description}
\end{prop}
\begin{proof}
(i) is a direct consequence of (\ref{rateconv}).

To prove (ii), we note the identity $1-(1-a)^{j+1}=(j+1)\int_0^a(1-u)^jdu$ and then write
\begin{multline*}
\left|\tilde g(x,\theta)-\tilde g_a(x,\theta)\right|\leq\sum_{j\geq 1}\left(1-(1-a)^{j+1}\right)\left|P_\theta^jf_\theta(x)\right|\\
\leq C(\compact)|f_\theta|_{V^\beta}V^{\beta+\alpha\kappa}(x)\sum_{j\geq 1}\int_0^a(1-u)^jdu (1+j)^{-\kappa+1}\\
=C(\compact)|f_\theta|_{V^\beta}V^{\beta+\alpha\kappa}(x)\int_0^a\left\{\sum_{j\geq 1}(1-u)^j(1+j)^{-\kappa+1}\right\}du\\
\leq C(\compact)|f_\theta|_{V^\beta}V^{\beta+\alpha\kappa}(x)2^{1-\kappa}\int_0^a\zeta_{\kappa-1}(u)du.
\end{multline*}
Since $\kappa>1$ and $a>0$, the interchange of the summation and integral signs is permitted.
\end{proof}

\begin{rem}\label{remboundga}
One can check using Lemma \ref{lemseq} that for $\kappa>1$, $\int_0^a\zeta_{\kappa-1}(u)du\to 0$ as $a\to0$. Hence a direct consequence of Proposition \ref{propboundga} is that for any $\beta\in [0,1-2\alpha)$ ($\alpha<1/2$), any $\kappa\in (1,\alpha^{-1}(1-\beta)-1]$, there exists a finite constant $C(\compact)$ such that for any $(x,\theta)\in\Xset\times \compact$,
\begin{equation}\label{controlg}\left|\tilde g(x,\theta)\right|\leq C(\compact)|f_\theta|_{V^\beta}V^{\beta+\alpha\kappa}(x).\end{equation}
\end{rem}

\begin{prop}\label{propboundgaga}
Assume A\ref{A1}.
\begin{description}
\item [(i)]
For any $\kappa,\delta\geq 0$ with $\kappa+\delta\leq \alpha^{-1}(1-\beta)-1$, there exists a finite constant $C(\compact)$ such that for any $\theta,\theta'\in\compact$, $x\in\Xset$ and $a\in (0,1/2]$
\begin{multline*}\left|\tilde g_a(x,\theta)-\tilde g_a(x,\theta')\right|\leq C(\compact)\sup_{\theta\in\compact}|f_\theta|_{V^\beta}\zeta_\kappa(a)\left(\zeta_\delta(a)D_{\beta+\alpha\delta}(\theta,\theta')+|f_\theta-f_{\theta'}|_{V^\beta}\right)
V^{\beta+\alpha(\kappa+\delta)}(x).\end{multline*}
\item [(ii)] Assume $\alpha<1/2$. For any  $\beta\in[0,1-2\alpha)$, any $\kappa\geq 0$, $\delta> 1$ with $\kappa+\delta\leq \alpha^{-1}(1-\beta)-1$, There exist a finite constant $C(\compact)$ such that for any  $x\in\Xset$, $\theta,\theta'\in\compact$ and any $a\in (0,1/2]$
\begin{multline*}\left|\tilde g(x,\theta)-\tilde g(x,\theta')\right|\leq C(\compact)\sup_{\theta\in\compact}|f_\theta|_{V^\beta}\left(\int_0^a\zeta_{\delta-1}(u)du+\zeta_\kappa(a)|f_\theta-f_{\theta'}|_{V^\beta}
\right.\\
\;\;\left.+ \zeta_\kappa(a)D_{\beta+\alpha\delta}(\theta,\theta')\right)V^{\beta+\alpha(\kappa+\delta)}(x).\end{multline*}
\end{description}
\end{prop}
\begin{proof}We have
\[\tilde g_a(x,\theta)-\tilde g_a(x,\theta')=\sum_{j\geq 0}(1-a)^{j+1}\left(\bar P_\theta^jf_\theta(x)-\bar P^j_{\theta'}f_\theta(x)\right)\;+\;\sum_{j\geq 0}(1-a)^{j+1}\bar P^j_{\theta'}\left(f_\theta(x)-f_{\theta'}(x)\right).\]

By Proposition \ref{propboundga} (i) we bound the second term in the rhs as follows.
\begin{equation}
\label{proofboundgagaeq1}\left|\sum_{j\geq 0}(1-a)^{j+1}\bar P^j_{\theta'}\left(f_\theta(x)-f_{\theta'}(x)\right)\right|\leq C(\compact)|f_\theta-f_{\theta'}|_{V^\beta}\zeta_\kappa(a)V^{\beta+\alpha\kappa}(x).\end{equation}
The first term in the rhs can be rewritten as
\[\sum_{j\geq 0}(1-a)^{j+1}\left(\bar P_\theta^jf_\theta(x)-\bar P^j_{\theta'}f_\theta(x)\right)\;=\;\sum_{j\geq 1}(1-a)^{j+1}\sum_{l=0}^{j-1}\bar P_\theta^l(\bar P_\theta-\bar P_{\theta'})\bar P_{\theta'}^{j-l-1}f_\theta(x).\]
From (\ref{rateconv}) of A\ref{A2} with $\kappa=\delta$, we have $|\bar P^l_{\theta'}f_\theta(x)|\leq C(\compact)\sup_{\theta\in\compact}|f_\theta|_{V^\beta}(1+l)^{-\delta}V^{\beta+\alpha\delta}(x)$ for all $l\geq 0$. Combined with the definition of $D_{\beta+\alpha\delta}$, we get
\[\left|(\bar P_\theta-\bar P_{\theta'})\bar P_{\theta'}^{j-l-1}f_\theta(x)\right|\leq C(\compact) \sup_{\theta\in\compact}|f_\theta|_{V^\beta} (j-l)^{-\delta}D_{\beta+\alpha\delta}(\theta,\theta')V^{\beta+\alpha\delta}(x).\]
Another application of A\ref{A2}-(\ref{rateconv})  then yields for any $\kappa\in [0,\alpha^{-1}(1-\beta)-1-\delta]$
\begin{equation*}\label{eq:propboundgagaeq1}\left|\bar P_\theta^jf_\theta(x)-\bar P_{\theta'}^jf_{\theta}(x)\right|\leq C(\compact)\sup_{\theta\in\compact}|f_\theta|_{V^\beta}D_{\beta+\alpha\delta}(\theta,\theta')V^{\beta+\alpha(\kappa+\delta)}(x)
\sum_{l=0}^{j-1}(1+l)^{-\kappa}(j-l)^{-\delta}.
\end{equation*}
It follows that
\begin{multline}\label{proofboundgagaeq2}
\left|\sum_{j\geq 0}(1-a)^{j+1}\left(\bar P_\theta^jf_\theta(x)-\bar P^j_{\theta'}f_\theta(x)\right)\right|\nonumber\\
\leq  C(\compact)\sup_{\theta\in\compact}|f_\theta|_{V^\beta}D_{\beta+\alpha\delta}(\theta,\theta')V^{\beta+\alpha(\kappa+\delta)}(x)\sum_{j\geq 1}(1-a)^{j+1}\sum_{l=0}^{j-1}(1+l)^{-\kappa}(j-l)^{-\delta}\\
\leq C(\compact)\sup_{\theta\in\compact}|f_\theta|_{V^\beta}\zeta_\kappa(a) \zeta_\delta(a) D_{\beta+\alpha\delta}(\theta,\theta')V^{\beta+\alpha(\kappa+\delta)}(x)\nonumber.\end{multline}
Combining this with (\ref{proofboundgagaeq1}) gives part (i).

To prove (ii), we write $|\tilde g(x,\theta)-\tilde g(x,\theta')|\leq |\tilde g_a(x,\theta)-\tilde g(x,\theta)| +|\tilde g_a(x,\theta)-\tilde g_a(x,\theta')| +|\tilde g_a(x,\theta')-\tilde g(x,\theta')|$. Part (i) gives
\[|\tilde g_a(x,\theta)-\tilde g_a(x,\theta')|\leq C(\compact)\sup_{\theta\in\compact}|f_\theta|_{V^\beta}\zeta_\kappa(a)
\left(D_{\beta+\alpha\delta}(\theta,\theta')+|f_\theta-f_{\theta'}|_{V^\beta}\right)V^{\beta+\alpha(\delta+\kappa)}(x).\] Then we use $\delta>1$ and Part (ii) of Proposition \ref{boundga}, to get
\[|\tilde g_a(x,\theta)-\tilde g(x,\theta)| + |\tilde g_a(x,\theta')-\tilde g(x,\theta')|\leq C(\compact) \sup_{\theta\in\compact}|f_\theta|_{V^\beta} \int_0^a\zeta_{\delta-1}(u)du V^{\beta+\alpha\delta}(x).\]
The conclusion follows.
\end{proof}
%

\subsection{Modulated moments}\label{sec:moments}
In this section, $\compact$ is an arbitrary compact subset of $\Theta$, $(x,\theta)\in\Xset\times\compact$ and $l\geq 0$ an integer. We consider the nonhomogeneous Markov chain $\{(\tilde X_n,\tilde\theta_n),\;n\geq 0\}$ with initial distribution $\delta_{x,\theta}$ and transition kernels $P_l\left(n;(x_1,\theta_1),(dx',d\theta')\right)=\bar R\left(l+n;(x_1,\theta_1),(dx',d\theta')\right)$. Its distribution and expectation operator are denoted respectively by $\PP_{x,\theta}^{(l)}$ and $\PE_{x,\theta}^{(l)}$. The key property that we will use here is (\ref{adapt22}) which, as we have seen, is a consequence of (\ref{adapt1}). The first two propositions below are easy modifications of similar results proved in \cite{atchadeetfort08}.

\begin{prop}\label{prop:useful}
Assume A\ref{A1}. There exists a finite constant $C(\compact)$ such that for any $(x,\theta)\in\Xset\times\compact$, $l,n\geq 1$,
\begin{enumerate}[(i)]
\item \label{prop:Useful1} for any $0 \leq \beta \leq 1$,
\[
\PE_{x,\theta}^{(l)}\left[ V^\beta(\tilde X_n)\un_{\{\tauout_\compact>n-1\}}\right] \leq C(\compact)n^\beta\;V^\beta(x) \eqsp.
\]
\item \label{prop:Useful2} for any $0 \leq \beta \leq 1-\alpha$
\[
\PE_{x,\theta}^{(l)}\left[\sum_{k=1}^n V^\beta(\tilde X_k)\un_{\{\tauout_\compact>k-1\}}\right] \leq C(\compact)n\;V^{\beta+\alpha}(x) \eqsp.
\]
\end{enumerate}
\end{prop}

\begin{prop}\label{lem:sumdrift}
Assume A\ref{A1}. Let $\{r_n,\;n\geq 0\}$ be a non-increasing sequence of positive numbers. For $\beta\in[0,1-\alpha]$, there exists a finite constant $C(\compact)$ such that for any $(x,\theta)\in\Xset\times\compact$, $1\leq n< N$
\begin{multline*}\PE_{x,\theta}^{(l)}\left[\sum_{k=n}^{N-1} r_{k+1}V^{\beta}(\tilde X_k)\un_{\{\tauout_K>k-1\}}\right]\leq C(\compact)\left(r_n\PE_{x,\theta}^{(l)}\left(V^{\beta+\alpha}(\tilde X_n)\un_{\{\tauout_K>n-1\}}\right)+ \sum_{k=n}^N r_{k+1}\right).\end{multline*}
\end{prop}

The next proposition gives a general standard bound on moments of martingales as a consequence of the Burkholder's inequality.
\begin{prop}\label{prop:boundmartingale}
Let $M_{n}=\sum_{k=1}^nD_{k}$, $n\geq 1$ be a martingale such that $\PE\left(|D_{k}|^p\right)<\infty$ for some $p>1$. Then
\[\PE\left[\left|M_{n}\right|^p\right]\leq Cn^{\max(1,p/2)-1}\sum_{k=1}^n\PE\left(|D_{k}|^p\right),\]
for $C=(18pq^{1/2})^p$, $p^{-1}+q^{-1}=1$.
\end{prop}

\subsection{A Weak law of large numbers}\label{sec:llnwrep}
We fix $l\geq 0$ integer, $\compact$ a compact subset of $\Theta$ and $(x,\theta)\in\Xset\times\compact$. This section deals with the weak law of large numbers for the nonhomogeneous Markov chain $\{(\tilde X_n,\tilde\theta_n),\;n\geq 0\}$ with initial distribution $\delta_{x,\theta}$ and transition kernels $P_l\left(n;(x_1,\theta_1),(dx',d\theta')\right)=\bar R\left(l+n;(x_1,\theta_1),(dx',d\theta')\right)$.

\begin{prop}\label{prop:sllnpwrep}Assume A\ref{A1}. Let $\beta\in [0,1-\alpha)$ and $f_\theta\in\L_{V^\beta}$ a class of functions such that $\theta\to f_\theta(x)$ is a measurable map, $\pi(f_\theta)=0$ and $\sup_{\theta\in\compact}|f_\theta|_{V^\beta}<\infty$. Suppose also that there exist $\epsilon>0$, $\kappa>0$ such that $\beta+\alpha\kappa<1-\alpha$ and
\begin{equation}\label{diminishlln}
\PE_{x,\theta}^{(l)}\left[\sum_{k\geq 1}\un_{\{\tauout_\compact>k\}}k^{-1+\epsilon}\left(D_\beta(\tilde\theta_k,\tilde\theta_{k-1})
+|f_{\tilde\theta_k}-f_{\tilde\theta_{k-1}}|_{V^\beta}\right)V^{\beta+\alpha\kappa}(\tilde X_k)\right]<\infty.\end{equation}
Then $n^{-1}\un_{\{\tauout_\compact>n\}}\sum_{k=1}^nf_{\tilde\theta_{k-1}}(\tilde X_k)$ converges to zero in $\PP_{x,\theta}^{(l)}$-probability.
\end{prop}
\begin{proof}
Define $\tilde H_{a,\theta}(x,y)=\tilde g_a(y,\theta)-P_{\theta}\tilde g_a(x,\theta)$ and $S_n=\sum_{k=1}^n \un_{\{\tauout_\compact>k-1\}}f_{\tilde\theta_{k-1}}(\tilde X_k)$. Note that $\un_{\{\tauout_\compact>n\}}n^{-1}\sum_{k=1}^nf_{\tilde\theta_{k-1}}(\tilde X_k)=\un_{\{\tauout_\compact>n\}}n^{-1}S_n$.  Then we use (\ref{approxpoisson}) to re-write $S_n$ as:
\begin{eqnarray*}
S_n&=&\sum_{k=1}^n\un_{\{\tauout_\compact>k-1\}}\tilde H_{a_n,\tilde\theta_{k-1}}(\tilde X_{k-1},\tilde X_k) \; + \; \left((1-a_n)^{-1}-1\right)\sum_{k=1}^n\un_{\{\tauout_\compact>k-1\}}\tilde g_{a_n}(\tilde X_k,\tilde\theta_{k-1})\\
&&\;+\; \left( P_{\theta_0}\tilde g_{a_n}(\tilde X_0,\theta_0)-\un_{\{\tauout_\compact>n\}}P_{\tilde\theta_n}\tilde g_{a_n}(\tilde X_n,\tilde\theta_n)\right)\\
&&\;+\; \sum_{k=1}^n\un_{\{\tauout_\compact>k\}}\left(P_{\tilde\theta_{k}}\tilde g_{a_n}(\tilde X_{k},\tilde\theta_{k})-P_{\tilde\theta_{k-1}}\tilde g_{a_n}(\tilde X_{k},\tilde\theta_{k-1})\right)
+\;\sum_{k=1}^n\un_{\{\tauout_\compact=k\}}P_{\tilde\theta_{k-1}}\tilde g_{a_n}(\tilde X_{k},\tilde\theta_{k-1}).\end{eqnarray*}

We take  $a_n\propto n^{-\rho}\in (0,1/2]$ where $\rho>0$ is such that $\rho(1-\kappa)<\min\left(0.5,\alpha,1-p^{-1}\right)$ where $p=(1-\alpha)(\beta+\alpha\kappa)^{-1}>1$; and $\rho(2-\kappa)<\epsilon$, where $\kappa$ and $\epsilon$ are as in (\ref{diminishlln}). First, we notice that
\[\un_{\{\tauout_\compact>n\}}\sum_{k=1}^n\un_{\{\tauout_\compact=k\}}P_{\tilde\theta_{k-1}}g_{a_n}(\tilde X_{k},\tilde\theta_{k-1})=0.\]

Then we consider the term $M_{n,k}\eqdef \sum_{j=1}^k\tilde \un_{\{\tauout_\compact>k-1\}}H_{a_n,\tilde\theta_{j-1}}(\tilde X_{j-1},\tilde X_j)$. Clearly, $\{(M_{n,k},\F_k)\}$ is a martingale array. Applying Proposition \ref{propboundga} and Proposition \ref{prop:boundmartingale} (with $p=(1-\alpha)/(\beta+\alpha\kappa)>1$), we get
\begin{multline*}
\PE_{x,\theta}^{(l)}\left[|M_{n,n}|^p\right]\leq C(\compact)a_n^{p(\kappa-1)}n^{max(1,p/2)-1}\PE^{(l)}_{x,\theta}\left(\sum_{k=1}^n\un_{\{\tauout_\compact>k-1\}}V^{1-\alpha}(\tilde X_k)\right)\\
=O\left(n^{\rho p(1-\kappa)}n^{\max(1,p/2)}\right).\end{multline*}
By the choice of $\rho$, $\rho(1-\kappa)+\max(0.5,p^{-1})<1$ and we conclude that $M_{n,n}/n$ converges in $L^p$ to zero.

Define $R_n^{(1)}\eqdef\left((1-a_n)^{-1}-1\right)\sum_{k=1}^n\un_{\{\tauout_\compact>k-1\}}\tilde g_{a_n}(\tilde X_k,\tilde\theta_{k-1})$. Proposition \ref{propboundga} (i) implies that
\[\PE^{(l)}_{x,\theta}\left[n^{-1}|R_n^{(1)}|\right]\leq Ca_n^\kappa  n^{-1}\PE^{(l)}_{x,\theta}\left(\sum_{k=1}^n\un_{\{\tauout_\compact>k-1\}}V^{1-\alpha}(\tilde X_k)\right)=O\left(a_n^\kappa\right).\]
The rhs converges to zero since $a_n\to 0$ and $\kappa>0$.

We turn to $R_n^{(2)}\eqdef P_{\theta_0}\tilde g_{a_n}(\tilde X_0,\theta_0)-\un_{\{\tauout_\compact>n\}}P_{\tilde\theta_n}\tilde g_{a_n}(\tilde X_n,\tilde\theta_n)$. Again, by Proposition \ref{propboundga} (i), the drift condition in A\ref{A2}, and Proposition \ref{prop:useful} (i)
\begin{multline*}
\PE_{x,\theta}^{(l)}\left(n^{-1}|R_n^{(2)}|\right)\leq Cn^{-1}a_n^{-1+\kappa}\PE_{x,\theta}^{(l)}\left(V^{\beta+\alpha\kappa}(\tilde X_0)+\un_{\{\tauout_\compact>n\}}V^{\beta+\alpha\kappa}(\tilde X_n)\right)\\
=O\left(n^{-1+\beta+\alpha\kappa}a_n^{-1+\kappa}\right)=O\left(n^{-\alpha+\rho(1-\kappa)}\right).\end{multline*}
Given the assumption $\rho(1-\kappa)<\alpha$, it follows that $n^{-1}R_n^{(2)}$ converges in probability to zero.

We finally turn to $R_n^{(3)}\eqdef \sum_{k=1}^n\un_{\{\tauout_\compact>k\}}\left(P_{\tilde\theta_{k}}\tilde g_{a_n}(\tilde X_{k},\tilde\theta_{k})-P_{\tilde\theta_{k-1}}\tilde g_{a_n}(\tilde X_{k},\tilde\theta_{k-1})\right)$. By definition, $P_\theta \tilde g_a(x,\theta)-P_{\theta'}\tilde g_a(x,\theta')=f_{\theta'}(x)-f_{\theta}(x)+(1-a)^{-1}(\tilde g_a(x,\theta)-\tilde g_a(x,\theta'))$. By Proposition \ref{propboundgaga} (with $\delta=0$)  we have:
\begin{multline*}
\left|P_\theta \tilde g_{a_n}(x,\theta)-P_{\theta'}\tilde g_{a_n}(x,\theta')\right|\leq C(\compact)\sup_{\theta\in\compact}|f_\theta|_{V^\beta}a_n^{-2+\kappa}\left(D_\beta(\theta,\theta')
+|f_\theta-f_{\theta'}|_{V^\beta}\right)V^{\beta+\alpha\kappa}(x)\\
\leq C(\compact)\sup_{\theta\in\compact}|f_\theta|_{V^\beta}n^\epsilon\left(D_\beta(\theta,\theta')
+|f_\theta-f_{\theta'}|_{V^\beta}\right)V^{\beta+\alpha\kappa}(x)\end{multline*}
Therefore Kronecker's lemma and (\ref{diminishlln}) implies that $n^{-1}R_n^{(3)}$ converge almost surely to zero.
\end{proof}

The next result will be useful in proving the central limit theorem. We take $f\in\L_{V^\beta}$ and let $g_a$ be the resolvent associated with $f$ and $H_{a,\theta}(x,y):=g_a(y,\theta)-P_\theta g_a(x,\theta)$. We will show in the next lemma that $n^{-1/2}\un_{\{\tauout_\compact>n\}}\sum_{k=1}^n f(\tilde X_k)$ behaves like the martingale array $n^{-1/2}\sum_{k=1}^n \un_{\{\tauout_\compact>k-1\}}H_{a_n,\tilde\theta_{k-1}}(\tilde X_{k-1},\tilde X_k)$ as $n\to\infty$ for some well chosen sequence $\{a_n,\;n\geq 0\}$.

\begin{lemma}\label{prop:llnhalf}Assume A\ref{A1} with $\alpha<1/2$ and let $\compact$ a compact subset of $\Theta$. Let $\beta\geq 0$ such that $2(\beta+\alpha)<1$ and $f\in\L_{V^\beta}$ such that $\pi(f)=0$. Let $\kappa>1$, $\delta\in (0,1)$ be such that  $2\beta+\alpha(\kappa+\delta)<1-\alpha$. Take $\rho\in(1/2, 1/(2-\delta)]$ and let $\{a_n,\;n\geq 0\}$ be a sequence of positive numbers such that $a_n\in (0,1/2]$, $a_n\propto n^{-\rho}$. Suppose that
\begin{equation}\label{diminishllnhalf}
\PE^{(l)}_{x,\theta}\left[\sum_{k\geq 1}\un_{\{\tauout_\compact>k\}}k^{-1+\rho(2-\delta)}D_{\beta+\alpha\delta}(\tilde\theta_k,\tilde\theta_{k-1})V^{\beta+\alpha(\kappa+\delta)}(\tilde X_k)\right]<\infty.\end{equation}
For any $s\geq 0$, $n^{-1/2}\un_{\{\tauout_\compact>n\}}\sum_{k=1}^n\left(f(\tilde X_k)-H_{a_{n+s},\tilde\theta_{k-1}}(\tilde X_{k-1},\tilde X_k)\right)$ converges to zero in $\PP_{x,\theta}^{(l)}$-probability.
\end{lemma}
\begin{proof}
Without any loss of generality, we assume that $\kappa$ also satisfies $\beta+\alpha\kappa<1/2$. For $s\geq 0$ arbitrary, define $S_{n,s}=\sum_{k=1}^n \un_{\{\tauout_\compact>k-1\}}\left(f(\tilde X_k)-H_{a_{n+s},\tilde\theta_{k-1}}(\tilde X_{k-1},\tilde X_k)\right)$. Note that \[\un_{\{\tauout_\compact>n\}}n^{-1/2}\sum_{k=1}^n\left(f(\tilde X_k)- H_{a_{n+s},\tilde\theta_{k-1}}(\tilde X_{k-1},\tilde X_k)\right)=\un_{\{\tauout_\compact>n\}}n^{-1/2}S_n.\]
 Then we use the approximate Poisson equation (\ref{approxpoisson}) to re-write $S_{n,s}$ as:
\begin{eqnarray*}
S_{n,s}&=&\left((1-a_{n+s})^{-1}-1\right)\sum_{k=1}^n\un_{\{\tauout_\compact>k-1\}}g_{a_{n+s}}(\tilde X_{k},\tilde\theta_{k-1})\\
 &&\; \;+ \left(P_{\theta_0} g_{a_{n+s}}(\tilde X_0,\theta_0)-\un_{\{\tauout_\compact>n\}}P_{\tilde\theta_n}g_{a_{n+s}}(\tilde X_n,\tilde\theta_n)\right)\\
 &&+ \; \sum_{k=1}^n\un_{\{\tauout_\compact>k\}}\left(P_{\tilde\theta_{k}}g_{a_{n+s}}(\tilde X_{k},\tilde\theta_{k})-P_{\tilde\theta_{k-1}} g_{a_{n+s}}(\tilde X_{k},\tilde\theta_{k-1})\right)\\
 &&\;\; + \sum_{k=1}^n\un_{\{\tauout_\compact=k\}}P_{\tilde\theta_{k-1}} g_{a_{n+s}}(\tilde X_{k},\tilde\theta_{k-1}).\end{eqnarray*}

Notice that $\un_{\{\tauout_\compact>n\}}\sum_{k=1}^n\un_{\{\tauout_\compact=k\}}P_{\tilde\theta_{k-1}}g_{a_{n+s}}(\tilde X_{k},\tilde\theta_{k-1})=0$.  For the rest, consider $R_n^{(1)}\eqdef \left(P_{\theta_0}g_{a_{n+s}}(\tilde X_0,\theta_0)-\un_{\{\tauout_\compact>n\}}P_{\tilde\theta_n}g_{a_{n+s}}(\tilde X_n,\tilde\theta_n)\right)$. By Proposition \ref{propboundga}, the choice $\kappa>1$, and by Proposition \ref{prop:useful} (i) we have
\[\PE_{x,\theta}^{(l)}\left(|R_n^{(1)}|\right)\leq C(\compact)\PE_{x,\theta}^{(l)}\left(V^{\beta+\alpha\kappa}(x)+V^{\beta+\alpha\kappa}(\tilde X_n)\un_{\{\tauout_\compact>n\}}\right)=O\left(n^{\beta+\alpha\kappa}\right).\]
Since $\beta+\alpha\kappa<1/2$ we deduce that $n^{-1/2}R_n^{(1)}\to 0$ in probability.

Now take $R_n^{(2)}\eqdef \left(1-a_{n+s})^{-1}-1\right)\sum_{k=1}^n\un_{\{\tauout_\compact>k-1\}}g_{a_{n+s}}(\tilde X_k,\tilde\theta_{k-1})$. We can apply Proposition \ref{propboundga} to obtain
\[|R_n^{(2)}|\leq C(\compact) a_{n+s}\sum_{k=1}^n \un_{\{\tauout_\compact>k-1\}}V^{\beta+\alpha\kappa}(\tilde X_{k})\]
and by Proposition \ref{prop:useful} (ii), $\PE_{x,\theta}^{(l)}\left(n^{-1/2}|R_n^{(2)}|\right)=O\left(n^{1/2}a_n\right)$. By assumption $a_n\propto n^{-\rho}$ with $\rho>1/2$, thus $n^{-1/2}R_n^{(2)}$ converges in probability to zero.

Finally, we consider $R_n^{(3)}\eqdef \sum_{k=1}^n\un_{\{\tauout_\compact>k\}}\left(P_{\tilde\theta_{k}}g_{a_{n+s}}(\tilde X_{k},\tilde\theta_{k})-P_{\tilde\theta_{k-1}}g_{a_{n+s}}(\tilde X_{k},\tilde\theta_{k-1})\right)$. By definition,
\[P_\theta g_a(x,\theta)-P_{\theta'}g_a(x,\theta')=(1-a)^{-1}(g_a(x,\theta)-g_a(x,\theta')),\]
and by Proposition \ref{propboundga} applied with $\kappa>1$ and $\delta>0$, $\left|P_\theta g_a(x,\theta)-P_{\theta'}g_a(x,\theta')\right|
\leq C(\compact)\zeta_\delta(a)D_{\beta+\alpha\delta}(\theta,\theta')V^{\beta+\alpha(\kappa+\delta)}(x)$ so that
\[|n^{-1/2}R_n^{(3)}|\leq C(\compact)n^{-1/2}n^{\rho(1-\delta)}\sum_{k=1}^n\un_{\{\tauout_\compact>k-1\}}D_{\beta+\alpha\delta}(\tilde\theta_k,\tilde\theta_{k-1})V^{\beta+\alpha(\kappa+\delta)}(\tilde X_k).\]
By assumption $n^{-1/2}n^{\rho(1-\delta)}=o(n^{-1+\rho(2-\delta)})$. Kronecker's lemma and (\ref{diminishllnhalf}) then gives that $n^{-1/2}R_n^{(3)}$ converges to $0$ with probability one.
\end{proof}

\subsection{Connection with the adaptive MCMC process}\label{sec:link}
In this section we give a number of results that connects the non-homogeneous Markov chain $\{(\tilde X_n,\tilde \theta_n),\;n\geq 0\}$ with the adaptive MCMC process $\{(X_n,\theta_n,\nu_n,\xi_n),\;n\geq 0\}$ defined in Section \ref{sec:defalgo}. This will allow us to transfer the limit results established above to the adaptive chain.

We introduce the sequence of stopping times associated with the adaptive chain
\[T_0=0\;\;\;\;\; T_{j+1}\eqdef\inf\{k>T_j,\; \xi_k=0\},\;\; k\geq 1,\]
with the convention that $\inf\emptyset=\infty$. Also define
\[\nu_\infty\eqdef \sup_{k\geq 0}\nu_k.\]
\begin{lemma}
\label{lemma:FiniteProj}
If A\ref{A2} hold then $\cPP\left( T_{\nu_\infty}<\infty\right)=1$.
\end{lemma}
\begin{proof}A\ref{A2} states that $\cPP\left(\nu_\infty<\infty\right)=1$. Thus under A\ref{A2}
\[\cPP\left(T_{\nu_\infty} =
      +\infty \right) =  \sum_{j \geq 0}
    \cPP_\star\left(T_{j} = + \infty, \nu_\infty =j
    \right) =0,\]
the last equality follows from the fact that on the set $\{T_j =
    +\infty\}$, $\sup_{k \geq 0} \nu_k \leq j-1$. Hence,
    $\cPP\left(T_{\nu_\infty} < +\infty \right) =1$.
\end{proof}

The following is Lemma 4.1 of \cite{AMP05}.
\begin{prop}
\label{prop:EquivProba}
For any $n \in \nset$, any n-uplet $(t_1, \cdots, t_n)$, any bounded measurable functions
$\{f_k, k \leq n \}$ and for any $(x,\theta,j)\in\Xset\times\compact_j\times\nset$,
\[\cPE_{x,\theta,j,0} \left[ \prod_{k=1}^n f_k(X_{t_k},\theta_{t_k})
  \un_{\{T_1>t_n\}} \right] = \PE_{x,\theta}^{(j)}
\left[ \prod_{k=1}^n f_k(\tilde X_{t_k},\tilde\theta_{t_k}) \un_{\{\tauout_{\compact_j}>t_n \}}\right]
\]
\end{prop}

One can obtain the finiteness of moments of the adaptive chain as in the  following lemma.
\begin{lemma}\label{lem:momentV}Let $\tilde W_n=W(\tilde X_n,\tilde \theta_n,\tilde X_{n+1})$ be a sequence of random variables such that for all $l,k\leq n$,
\[c_{k}^{(l)}:=\sup_{(x,\theta)\in \Xset_0\times \Theta_0}\PE_{x,\theta}^{(l)}\left[\tilde W_k\textbf{1}_{\{\tauout_{\compact_l}>k\}}\right]<\infty .\]
Then $\cPE\left(W(X_n,\theta_n,X_{n+1})\right)$ is finite.\end{lemma}
\begin{proof}
Denote $W_n=W(X_n,\theta_n,X_{n+1})$. We have
\begin{eqnarray*}\cPE\left[W_n\right]&=&\sum_{j=0}^n\sum_{s=j}^n\cPE\left[W_n
\textbf{1}_{\{\nu_n=j\}}\textbf{1}_{\{T_j=s\}}\right]
=\sum_{j=0}^n\sum_{s=j}^n\cPE\left[W_n\textbf{1}_{\{T_j=s\}}\textbf{1}_{\{T_{j+1}>s+(n-s)\}}\right]\\
&=&\sum_{j=0}^n\sum_{s=j}^n\cPE\left[\textbf{1}_{\{T_j=s\}}\cPE_{X_s,\theta_s,j,0}\left(W(X_{n-s},\theta_{n-s},X_{n+1-s})
\textbf{1}_{\{T_1>n-s\}}\right)\right]\\
&=&\sum_{j=0}^n\sum_{s=j}^n\cPE\left[\textbf{1}_{\{T_j=s\}}\PE_{X_s,\theta_s}^{(j)}\left(\tilde W_{n-s}
\textbf{1}_{\{\tauout_{\compact_{j}}>n-s\}}\right)\right]\leq \sum_{j=0}^n\sum_{s=j}^nc_{n-s}^{(j)}\cPP\left(T_j=s\right)<\infty.\end{eqnarray*}
The last equality uses Proposition \ref{prop:EquivProba}.
\end{proof}

In very general terms, the next result shows that a weak law of large numbers for the re-projection free process $\{(X_n,\theta_n),\;n\geq 0\}$ implies a similar result from the adaptive chain.

\begin{lemma}\label{prop:conv}Assume A\ref{A2}. Let $\{\tilde W_{n,k},\; 1\leq k\leq n\}$ be a triangular array of random variables of the form $\tilde W_{n,k}=W_{n}(\tilde \theta_{k-1},\tilde X_{k-1},\tilde \theta_{k},\tilde X_k)$ for some measurable functions $W_n:\Theta\times \Xset\times\Theta\times \Xset\to \rset$. Let $\{b_n,\;n\geq 1\}$ a non-increasing sequence of positive number with $\lim_{n\to\infty}b_n=0$. Suppose that for any $k\geq 1$, $\sup_{n\geq 1}|W_n(\theta_{k-1},X_{k-1},\theta_k,X_k)|<\infty$ $\cPP$-a.s. and for all $l\geq 0$, $s\geq 0$, $(x,\theta)\in \Xset_0\times \compact_l$ and $\delta>0$
\[
\lim_{n\to\infty}\PP_{x,\theta}^{(l)}\left[b_n\un_{\{\tauout_{\compact_l}>n\}}\left|\sum_{k=1}^{n}\tilde W_{n+s,k}\right|>\delta\right]=0,\]
then $b_n\sum_{k=1}^nW_n(\theta_{k-1},X_{k-1},\theta_k,X_k)$ converges to zero in $\cPP$-probability as $n\to\infty$.
\end{lemma}
\begin{proof}
The idea of the proof is similar to the proof of Proposition 6 of \cite{andrieuetal06}. Write $W_{n,k}=W_n(\theta_{k-1},X_{k-1},\theta_k,X_k)$. As shown above A\ref{A2} implies that $T_{\nu_\infty}$ is finite  $\cPP$-a.s. With the convention that $\sum_a^b\cdot=0$ if $a>b$, we write
;\begin{eqnarray*}b_n\sum_{k=1}^nW_{n,k}&=&b_n\sum_{k=1}^{n\wedge T_{\nu_\infty}}W_{n,k}+b_n\sum_{k=T_{\nu_\infty}+1}^nW_{n,k}\\
&=&b_nS_n^{(1)}+b_nS_n^{(2)}.\end{eqnarray*}
where $S_n^{(1)}=\sum_{k=1}^{n\wedge T_{\nu_\infty}}W_{n,k}$ and $S_n^{(2)}=\sum_{k=T_{\nu_\infty}+1}^nW_{n,k}$. Since $\sup_{n\geq 1}|W_{n,k}|$ and  $T_{\nu_\infty}$ are finite  $\cPP$-a.s., we deduce that $|S_n^{(1)}|\leq \sum_{k=1}^{T_{\nu_\infty}}\sup_{n\geq 1}|W_{n,k}|$ is also finite $\cPP$-a.s. Therefore $b_nS_n^{(1)}$ converges to zero $\cPP$-a.s..

Take $\epsilon>0$. From Lemma \ref{lemma:FiniteProj}, we can find $L_2>L_1>0$ such that $\cPP\left[\nu_\infty\geq L_1\right]+\cPP\left[T_{\nu_\infty}\geq L_2\right]\leq \epsilon$. For any $\delta>0$ and $n>L_2$, we have:
\[\cPP\left[b_n|S_n^{(2)}|\geq \delta\right]\leq \sum_{l=0}^{L_1}\sum_{s=0}^{L_2}\cPP\left[b_n|S_n^{(2)}|>\delta, T_l=s,\nu_\infty=l\right]+\epsilon.\]
We then observe that the event
\[\left\{b_n|S_n^{(2)}|>\delta, T_l=s,\nu_\infty=l\right\}\subseteq\left\{b_{n-s}|\sum_{k=T_l+1}^{T_l+n-s}W_{n,k}|\un_{\{T_{l+1}>T_l+n-s\}}>\delta\right\}.\]
Therefore by conditioning on $\F_{T_l}$, we get:
\begin{multline*}
\cPP\left[b_n|S_n^{(2)}|>\delta, T_l=s,\nu_\infty=l\right]\\
\;\;\;\;\;\;\;\leq \cPE\left[\cPP_{X_{T_l},\theta_{T_l},l,0}\left(b_{n-s}\left|\sum_{k=1}^{n-s}W_{n,k}\right|\un_{\{T_1>n-s\}}>\delta\right)\right]\\
=\cPE\left[\PP^{(l)}_{\tilde X_{T_l},\tilde \theta_{T_l}}\left(b_{n-s}\left|\sum_{k=1}^{n-s}\tilde W_{n,k}\right|\un_{\{\tauout_{\compact_{l}}>n-s\}}>\delta\right)\right].\end{multline*}
The last equality follows from Proposition \ref{prop:EquivProba}. By assumption, the inner term in the last expectation above converges almost surely to zero.  It follows from Lebesgue's dominated convergence theorem that $\lim_{n\to\infty}\cPP\left(b_n|S_n^{(2)}|\geq \delta\right)\leq \epsilon$. Since $\epsilon>0$ is arbitrary, the results follows.
\end{proof}

\subsection{Proof of Theorem \ref{thm1}}\label{sec:proofthm1}
Since A\ref{A1} and (\ref{diminish}) hold, we can apply Proposition \ref{prop:sllnpwrep} which implies that  \[\lim_{n\to\infty}\PP_{x,\theta}^{(l)}\left[n^{-1}\un_{\{\tauout_{\compact_l}>n\}}
\left|\sum_{k=1}^{n}f_{\theta_{k-1}}(X_k)\right|>\delta\right]=0,\] for any $\delta>0$, $l\geq 0$ and $(x,\theta)\in \Xset\times \compact_l$. Theorem \ref{thm1} then follows from Lemma \ref{prop:conv}.

\subsection{Proof of Theorem \ref{thm3}}\label{sec:proofthm3}
Throughout the proof, we take $\kappa>1, \delta\in(0,1)$, $\rho\in (1/2,(2-\delta)^{-1}]$ and $\{a_n,\;n\geq 0\}$ as in the statement of the theorem. Denote $S_n=\sum_{k=1}^nf(X_k)$. Without any loss of generality, we will assume that $|f|_{V^\beta}\leq 1$. We have
\begin{multline*}
S_n=\sum_{k=1}^nH_{a_n,\theta_{k-1}}(X_{k-1},X_k)\un_{\{\xi_{k-1}\neq 0\}}+\sum_{k=1}^nH_{a_n,\theta_{k-1}}(X_{k-1},X_k)\un_{\{\xi_{k-1}=0\}} \\
\;\;+ \sum_{k=1}^n\left(f(X_k)-H_{a_n,\theta_{k-1}}(X_{k-1},X_k)\right).\end{multline*}

By Theorem \ref{thm2}, $n^{-1/2}\sum_{k=1}^n\left(f(X_k)-H_{a_n,\theta_{k-1}}(X_{k-1},X_k)\right)$ converges in $\cPP$-probability to zero.

Note that $\xi_k=0$ signals a re-projection at time $k$. By Proposition \ref{propboundga} (i) applied with $\kappa>1$, \[\left|\sum_{k=1}^nH_{a_n,\theta_{k-1}}(X_{k-1},X_k)\un_{\{\xi_{k-1}=0\}}\right|\leq C(\compact_{\nu_\infty})\sum_{k=1}^{T_{\nu_\infty}}\left(V^{1-\alpha}(X_{k-1})+ V^{1-\alpha}(X_{k})\right), \;\;\cPP-\mbox{a.s.}\]
and the rhs is finite $\cPP$-a.s. We then conclude that $n^{-1/2}\sum_{k=1}^nH_{a_n,\theta_{k-1}}(X_{k-1},X_k)\un_{\{\xi_{k-1}=0\}}$ converges to zero $\cPP$-a.s..

Define $M_{n,k}=\sum_{j=1}^kD_{n,j}$, where $D_{n,j}=n^{-1/2}H_{a_n,\theta_{j-1}}(X_{j-1},X_j)\un_{\{\xi_{j-1}\neq 0\}}$. It is straightforward to see that $\{(M_{n,k},\F_k),\;1\leq k\leq n\}$ is a martingale array. We will show that

\begin{equation}\label{eq:TBS1}\{(M_{n,k},\check\F_k),\;1\leq k\leq n\}\;\;\; \mbox{is a square-integrable martingale array};\end{equation}
\begin{equation}\label{eq:TBS2}\lim_{n\to\infty} \sum_{j=1}^n \cPE\left(D_{n,j}^2\vert \check\F_{j-1}\right) =\sigma^2_\star(f),\;\;\;\left(\mbox{in $\cPP$-probab.}\right)\end{equation}
where
\begin{equation}\sigma^2_\star(f)\eqdef\int\pi(dx)\left(-f^2(x)+2f(x)g(x,\theta_\star)\right),\end{equation}
is finite $\cPP$-almost surely and that for any $\epsilon>0$,
\begin{equation}\label{eq:TBS3}\lim_{n\to\infty} \sum_{k=1}^n\cPE\left(D_{n,j}^2\un_{\{|D_{n,j}|\geq \epsilon\}}|\check\F_{j-1}\right)=0,\;\;\;\left(\mbox{in $\cPP$-probab.}\right)\end{equation}

By the central limit theorem for martingales (see e.g. \cite{halletheyde80}, Corollary 3.1), (\ref{eq:TBS1})-(\ref{eq:TBS3}) implies that $M_{n,n}$ converges weakly to $Z$ ($M_{n,n}\stackrel{w}{\to}Z$) where $Z$ is a random variable with characteristic function $\phi(t)=\cPE\left(e^{-\frac{1}{2}\sigma_\star^2(f)t^2}\right)$. This will end the proof.

\paragraph{\tt Proof of (\ref{eq:TBS1})}\;\; It suffices to show that for all $l\geq 0$, $k,n\geq 1$,
\[\sup_{(x,\theta)\in\Xset_0\times \compact_l}\PE_{x,\theta}^{(l)}\left(H_{a_n,\tilde \theta_{k-1}}^2(\tilde X_{k-1},\tilde X_k)\un_{\{\tauout_{\compact_l}>k-1\}}\right)<\infty\]
and to apply Lemma \ref{lem:momentV}. By Proposition \ref{propboundga} (i) (applied with both $\kappa>1$ and $\delta>0$), $\sup_{\theta\in\compact}|g_a(x,\theta)|^2\leq C(\compact)\zeta_{\delta}(a)V^{2\beta+\alpha(\kappa+\delta)}(x)\leq C(\compact)\zeta_{\delta}(a)V^{1-\alpha}(x)$ since by assumption $2\beta+\alpha(\kappa+\delta)\leq 1-\alpha$. Thus for any $l\geq 0$, $k,n\geq 1$ and $(x,\theta)\in\Xset\times \compact_l$,
\begin{multline*}
\PE_{x,\theta}^{(l)}\left(H_{a_n,\tilde \theta_{k-1}}^2(\tilde X_{k-1},\tilde X_k)\un_{\{\tauout_{\compact_l}>k-1\}}\vert \F_{k-1}\right)\leq \un_{\{\tauout_{\compact_l}>k-1\}}P_{\tilde \theta_{k-1}} g_{a_n}^2(\tilde X_{k-1},\tilde \theta_{k-1})\\
\leq  C(\compact_l)\zeta_{\delta}(a_n)\un_{\{\tauout_{\compact_l}>k-1\}} V^{1-\alpha}(\tilde X_{k-1}).\end{multline*}
From Proposition \ref{prop:useful} (i) we thus obtain
\[\sup_{(x,\theta)\in\Xset_0\times \compact_l}\PE_{x,\theta}^{(l)}\left(H_{a_n,\tilde \theta_{k-1}}^2(\tilde X_{k-1},\tilde X_k)\un_{\{\tauout_{\compact_l}>k-1\}}\right)\leq C(\compact_l)\zeta_\delta(a_n) k^{1-\alpha}\;\sup_{x\in\Xset_0}V^{1-\alpha}(x)<\infty.\]

\paragraph{\tt Proof of (\ref{eq:TBS2})}\;\;
\begin{multline*}
\cPE\left(D_{n,j}^2\vert \check\F_{j-1}\right)=\un_{\{\xi_{j-1}\neq 0\}}n^{-1}P_{\theta_{j-1}} H^2_{a_n,\theta_{j-1}}(X_{j-1})\\
=n^{-1}P_{\theta_{j-1}} H^2_{a_n,\theta_{j-1}}(X_{j-1})
 \;\; -\un_{\{\xi_{j-1}= 0\}}n^{-1}P_{\theta_{j-1}} H^2_{a_n,\theta_{j-1}}(X_{j-1}).\end{multline*}
The same argument as above shows that
\[n^{-1}\sum_{j=1}^n\un_{\{\xi_{j-1}= 0\}}P_{\theta_{j-1}} H^2_{a_n,\theta_{j-1}}(X_{j-1})\leq n^{-1}\zeta_\delta(a_n) C(\compact_{\nu_\infty})\sum_{j=1}^{T_{\nu_\infty}}V^{1-\alpha}(X_{j-1}),\]
 which converges almost surely to zero since $T_{\nu_\infty}$ is finite $\cPP$-almost surely, $\zeta_\delta(a_n)=O(n^{\rho(1-\delta)})$ and $\rho(1-\delta)<1/2$.

For the first term, we note that $P_\theta H_{a,\theta}^2(x,\theta)=P_\theta g_{a}^2(x,\theta)-\left(P_\theta g_a(x,\theta)\right)^2=P_\theta g_{a}^2(x,\theta)-\left((1-a)^{-1}g_a(x,\theta)-f(x)\right)^2$. We thus have the decomposition:
\[\frac{1}{n}\sum_{k=1}^nP_{\theta_k-1}H^2_{a_n,\theta_{k-1}}(X_{k-1})=\frac{1}{n}\sum_{i=1}^6T_n^{(i)}+\int \pi(dx)\left(-f^2(x)+2f(x)g(x,\theta_\star)\right),\]
where
\[T_n^{(1)}=\sum_{k=1}^nP_{\theta_{k-1}}g_{a_n}^2(X_{k-1},\theta_{k-1})-g_{a_n}^2(X_{k-1},\theta_{k-1}),\]
\[T_n^{(2)}=\left(1-(1-a_n)^{-2}\right)\sum_{k=1}^ng_{a_n}^2(X_{k-1},\theta_{k-1}),\]
\[T_n^{(3)}=2\left((1-a_n)^{-1}-1\right)\sum_{k=1}^nf(X_{k-1})g_{a_n}(X_{k-1},\theta_{k-1}),\]
\[T_n^{(4)}=2\sum_{k=1}^nf(X_{k-1})\left(g_{a_n}(X_{k-1},\theta_{k-1})-g(X_{k-1},\theta_{k-1})\right),\]
\[T_n^{(5)}=2\sum_{k=1}^n\int \pi(dx)f(x)\left(g(x,\theta_{k-1})-g(x,\theta_\star)\right).\]
\[T_n^{(6)}=\sum_{k=1}^n\left[-f^2(X_{k-1}) +2 f(X_{k-1})g(X_{k-1},\theta_{k-1})-\int \pi(dx)\left(-f^2(x)+2f(x)g(x,\theta_{k-1})\right)\right].\]

By assumption $n^{-1}T_n^{(1)}$ converges in $\cPP$-probability to zero. We will use the same technique to study the term $T_n^{(2)}$ to $T_n^{(5)}$. For example for $T_n^{(2)}$, the idea is to introduce its counterpart $\tilde T_{n,s}^{(1)}$ in the space of the re-projection free process $\{(\tilde X_n,\tilde \theta_n),\;n\geq 0\}$, to show that $\lim_{n\to\infty}\PP_{x,\theta}^{(l)}\left(|\tilde T_{n,s}^{(2)}|>\delta\right)=0$ for any $l\geq 0$, $\delta>0 $ and any $(x,\theta)\in\Xset_0\times\Theta_l$ and then to argue that $\lim_{n\to\infty}\cPP\left(|T_n^{(1)}|>\delta\right)=0$ for all $\delta>0$ using Lemma \ref{prop:conv}.

\begin{lemma}$n^{-1}\left(T_n^{(2)}+T_n^{(3)}\right)$ converges in probability to zero.
\end{lemma}
\begin{proof}
For $l,s\geq 0$, define
\begin{multline*}
\tilde T_{n,s}\eqdef\left(1-(1-a_{n+s})^{-2}\right)\un_{\{\tauout_{\compact_l}>n\}}\sum_{k=1}^ng_{a_{n+s}}^2(\tilde X_{k-1},\tilde \theta_{k-1}) \\
\;\;+  \left((1-a_{n+s})^{-1}-1\right)\un_{\{\tauout_{\compact_l}>n\}}\sum_{k=1}^nf(\tilde X_{k-1})g_{a_{n+s}}(\tilde X_{k-1},\tilde \theta_{k-1}).\end{multline*}
We show that for any $\mu>0$, and any $(x,\theta)\in\Xset\times \compact_l$, $\lim_{n\to\infty}\PP_{x,\theta}^{(l)}\left(n^{-1}|\tilde T_{n,s}|>\mu\right)=0$. Then we can apply Lemma \ref{prop:conv} to conclude that $n^{-1}T_n^{(1)}$ converges in $\cPP$-probability to zero.
As above, for any $(x,\theta)\in\Xset\times \compact_l$ and by Proposition \ref{propboundga} (i), we get
\begin{multline*}
\PE_{x,\theta}^{(l)}\left(|\tilde T_{n,s}|\right)\leq\\
 C(\compact_l)\left(\zeta_\delta(a_{n+s})+1\right)a_{n+s}
\PE_{x,\theta}^{(l)}\left(\sum_{k=1}^n\un_{\{\tauout_{\compact_l}>k-1\}}V^{2\beta+\alpha(\kappa+\delta)}(\tilde X_k)\right)=O\left(na_n^{\delta}\right).\end{multline*}
The rest of the proof follows from the usual bounds on the $V$-moments.
\end{proof}

\begin{lemma}$n^{-1}T_n^{(4)}$ converges in probability to zero.
\end{lemma}
\begin{proof}
For $l,s\geq 0$, define
\begin{multline*}
\tilde T_{n,s}^{(4)}\eqdef\un_{\{\tauout_{\compact_l}>n\}}\sum_{k=1}^nf(\tilde X_{k-1})\left(g_{a_{n+s}}(\tilde X_{k-1},\tilde \theta_{k-1})-g(\tilde X_{k-1},\tilde \theta_{k-1})\right)\\
\;\;=\un_{\{\tauout_{\compact_l}>n\}}\sum_{k=1}^n\un_{\{\tauout_{\compact_l}>k-1\}}\tilde f(X_{k-1})\left(g_{a_{n+s}}(\tilde X_{k-1},\tilde \theta_{k-1})-g(\tilde X_{k-1},\tilde \theta_{k-1})\right).
\end{multline*}
Again, for any $(x,\theta)\in\Xset\times \compact_l$ and by Proposition \ref{propboundga} (ii) we get
\[\PE_{x,\theta}^{(l)}\left(n^{-1}|\tilde T_{n,s}^{(4)}|\right)\leq C(\compact_l)a_{n+s}\zeta_{\kappa-1}(a_{n+1})n^{-1}\PE_{x,\theta}^{(l)}\left(\sum_{k=1}^n\un_{\{\tauout_{\compact_l}>k-1\}}V^{2\beta+\alpha\kappa}(\tilde X_k)\right)
=O\left(a_{n}\zeta_{\kappa-1}(a_{n})\right).\]
The rest of the proof is similar to the above upon noticing that for $\kappa>1$, $a\zeta_{\kappa-1}(a)\to 0$ as $ a\to 0$.
\end{proof}

\begin{lemma}$n^{-1}T_n^{(5)}$ converges $\cPP$-almost surely to zero.
\end{lemma}
\begin{proof}By Proposition \ref{propboundgaga} (ii), there exists a finite constant $C(\compact)$ such that for any $\theta,\theta'\in\compact$, $x\in\Xset$ and any $a\in(0,1/2]$
\[\left|g(x,\theta)-g(x,\theta')\right|\leq C(\compact) \left( a\zeta_{\kappa-1}(a)+a^{-1}D_{\beta+\alpha\kappa}(\theta,\theta')\right)V^{\beta+\alpha\kappa}(x).\]
Therefore
\[\left|\int\pi(dx)f(x)\left(g(x,\theta)-g(x,\theta')\right)\right|\leq C(\compact) \left( a\zeta_{\kappa-1}(a)+a^{-1}D_{\beta+\alpha\kappa}(\theta,\theta')\right)\pi\left(V^{2\beta+\alpha\kappa}\right).\]
Let $\epsilon>0$. Since $a\zeta_{\kappa-1}(a)\to 0$ as $a\to 0$, we can find $a_0\in(0,1/2]$ such that $a_0\zeta_{\kappa-1}(a_0)<\epsilon$.
Then for $\cPP$-almost every sample path
\begin{multline*}
\lim_{n\to\infty}\left|\int\pi(dx)f(x)\left(g(x,\tilde \theta_n)-g(x,\theta_\star)\right)\right|\\
\leq C(\compact_{\nu_\infty}) \lim_{n\to\infty}\left(\epsilon+a_0^{-1}D_{\beta+\alpha\kappa}(\tilde \theta_n,\theta_\star)\right)
\pi\left(V^{2\beta+\alpha\kappa}\right)= \epsilon C(\compact_{\nu_\infty})\pi\left(V^{2\beta+\alpha\kappa}\right).
\end{multline*}
Since $\epsilon>0$ is arbitrary and $\pi\left(V^{2\beta+\alpha\kappa}\right)<\infty$, we are finished.
\end{proof}

\begin{lemma}$n^{-1}T_n^{(6)}$ converges in probability to zero.
\end{lemma}
\begin{proof}
We would like to apply the law of large number (Theorem \ref{thm2}) to show that $n^{-1}T_n^{(6)}$ converges to zero. By Proposition \ref{propboundga} (ii), for any compact subset $\compact$ of $\Theta$, $\sup_{\theta\in\compact}|f^2+2fg_\theta|_{V^{2\beta+\alpha\kappa}}<\infty$ and $2\beta+\alpha\kappa<1-\alpha$. To check (\ref{diminish}), it is enough to find $\epsilon>0$ such that
\begin{equation}\label{eqintermed:TBS2}\PE_{x,\theta}^{(l)}
\left[\sum_{k\geq 1}k^{-1+\epsilon}|fg_{\tilde \theta_{k-1}}-fg_{\tilde \theta_k}|_{V^{2\beta+\alpha\kappa}}\un_{\{\tauout_{\compact_l}>k\}}
V^{2\beta+\alpha(\kappa+\delta)}(\tilde X_k)\right]<\infty.\end{equation}
But by Proposition \ref{propboundgaga} (ii), there exists a finite constant $C(\compact)$ such that for any $\theta,\theta'\in\compact$, $x\in\Xset$ and any $a\in(0,1/2]$
\[\left|f(\cdot) g(\cdot,\theta)-f(\cdot) g(\cdot,\theta')\right|_{V^{2\beta+\alpha\kappa}}\leq C(\compact) a\zeta_{\kappa-1}(a)+a^{-1}D_{\beta+\alpha\kappa}(\theta,\theta').\]
We let $a$ depend on $k$ by taking $a=a_k$, therefore
\begin{multline*}
\PE_{x,\theta}^{(l)}
\left[\sum_{k\geq 1}k^{-1+\epsilon}|fg_{\tilde \theta_{k-1}}-fg_{\tilde \theta_k}|_{V^{2\beta+\alpha\kappa}}\un_{\{\tauout_{\compact_l}>k\}}
V^{2\beta+\alpha(\kappa+\delta)}(\tilde X_k)\right]\\
\leq \PE_{x,\theta}^{(l)}
\left[\sum_{k\geq 1}k^{-1+\epsilon}a_k\zeta_{\kappa-1}(a_k)\un_{\{\tauout_{\compact_l}>k\}}
V^{1-\alpha}(\tilde X_k)\right]\\
 + \PE_{x,\theta}^{(l)}
\left[\sum_{k\geq 1}k^{-1+\epsilon}a_k^{-1}D_{\beta+\alpha\kappa}(\tilde \theta_{k-1},\tilde \theta_k)
\un_{\{\tauout_{\compact_l}>k\}}V^{2\beta+\alpha(\kappa+\delta)}(\tilde X_k)\right].\end{multline*}
We can then find $\epsilon>0$ such that $n^{\epsilon}a_n\zeta_{\kappa-1}(a_n)+n^{-1+\epsilon}a_n^{-1}=O(n^{-\epsilon})$ and (\ref{eqintermed:TBS2}) follows.
\end{proof}

\paragraph{\tt Proof of (\ref{eq:TBS3})}\;\; It is suffices to show that
\[\lim_{n\to\infty} n^{-1}\sum_{k=1}^n\int P_{\theta_{k-1}}(X_{k-1},dy)H^2_{a_n,\theta_{k-1}}(X_{k-1},y)\un_{\{|H{a_n,\theta_{k-1}}(X_{k-1},y)|\geq \epsilon\sqrt{n}\}}=0,\]
in $\cPP$-probability. We will do so by applying Lemma \ref{prop:conv} again. By a lemma due to Dvoretzky (Lemma 9 of \cite{andrieuetal06})
\[\int P_{\theta_{k-1}}(X_{k-1},dy)H^2_{a_n,\theta_{k-1}}(X_{k-1},y)\un_{\{|H_{a_n,\theta_{k-1}}(X_{k-1},y)|>\epsilon\sqrt{n}\}}\leq 4W_{n,k},\]
where
\[W_{n,k}=\int P_{\theta_{k-1}}(X_{k-1},dy)g_{a_n}^2(y,\theta_{k-1})\un_{\{|g_{a_n}(y,\theta_{k-1})|>\epsilon\sqrt{n}/2\}}.\]
It is thus enough to show that for any $s,l\geq 0$, any $(x,\theta)\in\Xset_0\times\compact_l$,
\[\lim_{n\to\infty}n^{-1}\sum_{k=1}^n\un_{\{\tauout_{\compact_l}>k-1\}}\tilde W_{n+s,k}=0,\;\;\mbox{(in $\PP_{x,\theta}^{(l)}$-probability)}.\]
Take $p>2$ such that $p(\beta+\alpha/2)<1-\alpha$. Then
\begin{multline*}
\PE_{x,\theta}^{(l)}\left(\un_{\{\tauout_{\compact_l}>k-1\}}\tilde W_{n+s,k}\right)
= \PE_{x,\theta}^{(l)}\left(\un_{\{\tauout_{\compact_l}>k-1\}}
\left|g_{a_{n+s}}(\tilde X_k,\tilde \theta_{k-1})\right|^2\un_{\{|g_{a_{n+s}}(\tilde X_k,\tilde \theta_{k-1})|>\epsilon\sqrt{n+s}/2\}}\right)\\
\leq (2/\epsilon)^{-p}(n+s)^{-p/2}\PE_{x,\theta}^{(l)}\left(\un_{\{\tauout_{\compact_l}>k-1\}}\left|g_{a_{n+s}}(\tilde X_k,\tilde \theta_{k-1})\right|^p\right)\\
\leq (2/\epsilon)^{-p}C(\compact_l)n^{-p/2}\left(\zeta_{1/2}(a_n)\right)^p
\PE_{x,\theta}^{(l)}\left(\un_{\{\tauout_{\compact_l}>k-1\}}V^{1-\alpha}(\tilde X_k)\right).
\end{multline*}
It follows that \[n^{-1}\PE_{x,\theta}^{(l)}\left(\sum_{k=1}^n\un_{\{\tauout_{\compact_l}>k-1\}}\tilde W_{n+s,k}\right)=O\left(n^{-p(1-\rho)/2}\right).\]
and since $\rho<1$, we are done.

\subsection{Proof of Proposition \ref{ExMC}}\label{sec:proofpropMC}
\begin{proof}
Denote $g_a(x)=\sum_{j\geq 0}(1-a)^{j+1}P^jf(x)$, $H_a(x,y)=g_a(y)-Pg_a(x)$ and write $g$ and $H$ respectively when $a=0$. Denote $L^2(\pi\!\times\!P)$ the $L^2$-space with respect to the joint measure $\pi(dx)P(x,dy)$ on $\Xset\times\Xset$. It is shown by \cite{mw00} (Proposition 1) that if $f\in L^2(\pi)$ and $\sum_{j\geq 1} j^{-1/2}\|P^jf\|_{L^2(\pi)}<\infty$ then there exists $H_\star\in L^2(\pi\!\times\!P)$ such that $\lim_{a\to 0}\|H_a-H_\star\|_{L^2(\pi\!\times\!P)}=0$.

Under (\ref{rateconv}) and with $f\in\L_{V^\beta}$, $\beta\in[0,1/2-\alpha)$, $\sum_{j\geq 1} j^{-1/2}\|P^jf\|_{L^2(\pi)}<\infty$ and thus there exists $H_\star\in L^2(\pi\!\times\!P)$ such that $\lim_{a\to 0}\|H_a-H_\star\|_{L^2(\pi\!\times\!P)}=0$. Moreover $\pi\!\times\!P(H_\star^2)=\pi(f(2g-f))$ (see e.g. \cite{holzmann05} for a derivation of this formula). From Proposition \ref{propboundga} (ii), we see that $H_\star=H$ ($\pi\!\times\!P$-a.s.). Note that $\pi\!\times\!P(H^2)=\pi\left(Pg^2-g^2+f(2g-f)\right)$ and $\pi(|f(2g-f)|)<\infty$ by Proposition \ref{propboundga} (i) and the fact that $\pi(V^{1-\alpha})<\infty$. Thus it follows from  $\pi\!\times\!P(H^2)<\infty$ and $\pi\!\times\!P(H^2)=\pi(f(2g-f))$ that $Pg^2-g^2$ is $\pi$-integrable and  $\pi(Pg^2-g^2)=0$.

On the other hand we have $P H_a^2(x)= P g_a^2(x)-(Pg_a(x))^2=P g_a^2(x)-g_a^2(x)+g_a^2(x)-(P g_a(x))^2$. Similarly $P H^2(x)=P g^2(x)-g^2(x)+g^2(x)-(P g(x))^2$. After some algebra we get
\begin{multline*}
\left(Pg_a^2(x)-g_a^2(x)\right)-\left(Pg^2(x)-g^2(x)\right)=PH_a^2(x)-PH^2(x) +2\left((1-a)^{-1}-1\right)f(x)g_a(x)\\
+2f(x)\left(g_a(x)-g(x)\right)-(\left(1-a)^{-1}+1\right)\left((1-a)^{-1}-1\right)g_a^2(x).\end{multline*}
We take $\kappa>1$ and $\delta>0$ such that $2\beta+\alpha(\kappa+\delta)<1-\alpha$ and apply Proposition \ref{propboundga} to get
\[\left|\left(Pg_a^2(x)-g_a^2(x)\right)-\left(Pg^2(x)-g^2(x)\right)\right|\leq \left|PH_a^2(x)-PH^2(x)\right| + C a^\delta V^{2\beta+\alpha(\kappa+\delta)}(x),\]
for some finite constant $C$. It follows that
\begin{multline}\label{eq:propA4}
\int\pi(dx)\left|\left(Pg_a^2(x)-g_a^2(x)\right)-\left(Pg^2(x)-g^2(x)\right)\right|\leq\\
 \|H_a-H\|_{L^2(\pi\times P)}^2 +2\|H_a-H\|_{L^2(\pi\times P)}\|H\|_{L^2(\pi\times P)} +C a^{\delta}\pi(V^{1-\alpha}).\end{multline}
Then we have
\begin{multline*}
n^{-1}\sum_{k=1}^ng_{a_n}^2(X_k)-Pg_{a_n}^2(X_{k})=n^{-1}\sum_{k=1}^ng^2(X_k)-Pg^2(X_k) \\
+n^{-1}\sum_{k=1}^ng_{a_n}^2(X_k)-Pg_{a_n}^2(X_{k}) -\left(g^2(X_k)-Pg^2(X_{k})\right).\end{multline*}
Since $\pi(|g^2-Pg^2|)<\infty$ and $\pi(g^2-Pg^2)=0$, the weak law of large numbers for Markov chains implies that $n^{-1}\sum_{k=1}^ng^2(X_k)-Pg^2(X_{k})$ converges in probability to zero. And
\begin{multline*}
n^{-1}\PE\left[\left|\sum_{k=1}^ng_{a_n}^2(X_k)-P_{\theta_\star}g_{a_n}^2(X_{k}) -g^2(X_k)-Pg^2(X_{k})\right|\right]\\
\leq \PE\left[\left|g_{a_n}^2(X_0)-P_{\theta_\star}g_{a_n}^2(X_{0})-g^2(X_0)-Pg^2(X_0)\right|\right]\end{multline*}
and the rhs converges to zero as a consequence of (\ref{eq:propA4}).
\end{proof}

\subsection{Proof of Proposition \ref{propMomCLT}}\label{sec:proofpropMomCLT}
\begin{proof}
Write
\begin{multline}\label{eqcorthm2}\sum_{k=1}^nP_{\theta_{k-1}}g_{a_{n}}^2(X_{k-1},\theta_{k-1})-g_{a_{n}}^2(X_{k-1},\theta_{k-1})
=\sum_{k=1}^nP_{\theta_{k-1}}g_{a_{n}}^2(X_{k-1},\theta_{k-1})-g_{a_{n}}^2(X_{k},\theta_{k-1})\\
+\sum_{k=1}^ng_{a_{n}}^2(X_{k},\theta_{k-1})-g_{a_{n}}^2(X_{k},\theta_{k})
+\left(g_{a_{n}}^2(X_{n},\theta_{n})-g_{a_{n}}^2(X_{0},\theta_{0})\right).\end{multline}

We first deal with the first term. For $l,s\geq 0$, Define
\[\tilde T_{n,s}^{(1)}=\un_{\{\tauout_{\compact_l}>n\}}\sum_{k=1}^nP_{\tilde \theta_{k-1}}g_{a_{n+s}}^2(\tilde X_{k-1},\tilde \theta_{k-1})-
g_{a_{n+s}}^2(\tilde X_{k},\tilde \theta_{k-1}).\]
We show that for any $\mu>0$, and any $(x,\theta)\in\Xset\times \compact_l$, $\lim_{n\to\infty}\PP_{x,\theta}^{(l)}\left(n^{-1}|\tilde T_{n,s}^{(1)}|>\mu\right)=0$. Then we can apply Lemma \ref{prop:conv} to conclude that $\sum_{k=1}^nP_{\theta_{k-1}}g_{a_{n}}^2(X_{k-1},\theta_{k-1})-g_{a_{n}}^2(X_{k},\theta_{k-1})$ converges in $\cPP$-probability to zero. We have
\[
\tilde T_{n,s}^{(1)}=\un_{\{\tauout_{\compact_l}>n\}}\sum_{k=1}^n\un_{\{\tauout_{\compact_l}>k-1\}}\left(P_{\tilde \theta_{k-1}}g_{a_{n+s}}^2(\tilde X_{k-1},\tilde \theta_{k-1})-g_{a_{n+s}}^2(\tilde X_{k},\tilde \theta_{k-1})\right)\] and $\PE_{x,\theta}^{(l)}\left[\un_{\{\tauout_{\compact_l}>k-1\}}\left(P_{\tilde \theta_{k-1}}g_{a_{n+s}}^2(\tilde X_{k-1},\tilde \theta_{k-1})-g_{a_{n+s}}^2(\tilde X_{k},\tilde \theta_{k-1})\right)\vert \F_{k-1}\right]=0$. Let $\kappa>1$ such that $2(\beta+\alpha\kappa)<2(\beta+\alpha)+\epsilon$. Set $p=(2(\beta+\alpha)+\epsilon)(2\beta+2\alpha\kappa)^{-1}$. By Proposition \ref{prop:boundmartingale}, we get \[\PE_{x,\theta}^{(l)}\left(|\tilde T^{(1)}_{n,s}|^p\right)\leq C(\compact_l)n^{1\vee 0.5 p}n^{-1}\PE_{x,\theta}^{(l)}\left(\sum_{j=1}^n\un_{\{\tauout_{\compact_l}>k-1\}}V^{2(\beta+\alpha)+\epsilon}(\tilde X_k)\right)=O\left(n^{1\vee 0.5p}\right),\]
by assumption.  Since  $p>1$, the result follows.

We use the same strategy to deal with the second term on the rhs of (\ref{eqcorthm2}). For $l,s\geq 0$, Define
\begin{multline*}
\tilde T_{n,s}^{(2)}\eqdef \un_{\{\tauout_{\compact_l}>n\}}\sum_{k=1}^ng_{a_{n+s}}^2(\tilde X_{k},\tilde \theta_{k-1})-g_{a_{n+s}}^2(\tilde X_{k},\tilde \theta_{k})
\\
=\un_{\{\tauout_{\compact_l}>n\}}\sum_{k=1}^n\un_{\{\tauout_{\compact_l}>k-1\}}\left(g_{a_{n+s}}(\tilde X_{k},\tilde \theta_{k-1})
-g_{a_{n+s}}(\tilde X_{k},\tilde \theta_{k})\right)\left(g_{a_{n+s}}(\tilde X_{k},\tilde \theta_{k-1}) + g_{a_{n+s}}(\tilde X_{k},\tilde \theta_{k})\right).\end{multline*}

We apply Proposition \ref{propboundga} (i) with $\kappa=\delta/2$ to get $\sup_{\theta,\theta'\in\compact_l}|g_{a}(x,\theta) + g_{a}(x,\theta')|\leq C(\compact_l)a^{-1+\delta/2}V^{\beta+\alpha\delta/2}(x)$. This together with Proposition \ref{propboundgaga} (i) (with $\kappa>1$ and $\delta/2>0$) gives:
\[\left|\tilde T_{n,s}^{(2)}\right|\leq C(\compact_l)\left(\zeta_{\delta/2}(a_{n+s})\right)^2\sum_{k=1}^n\un_{\{\tauout_{\compact_l}>k-1\}}
 D_{\beta+\alpha\delta/2}(\theta_{k-1},\theta_k)V^{2\beta+\alpha(\kappa+\delta)}(X_k)\]
$n^{-1}\left(\zeta_{\delta/2}(a_{n+s})\right)^2=O\left(n^{-1+\rho(2-\delta)}\right)$ then Kronecker's lemma and (\ref{diminishclt}) implies that $n^{-1}\tilde T_{n,s}^{(2)}$ converges in probability to zero.

For the last term on the rhs of (\ref{eqcorthm2}), define
\[\tilde T_{n,s}^{(3)}\eqdef\un_{\{\tauout_{\compact_l}>n\}}\left(g_{a_{n+s}}^2(\tilde X_{n},\tilde \theta_{n}) -g_{a_{n+s}}^2(\tilde X_{0},\tilde \theta_{0})\right).\]
Then with $\kappa_0>1$ such that $2(\beta+\alpha\kappa_0)<1$, we get the bound $\PE_{x,\theta}^{(l)}\left(n^{-1}|\tilde T_{n,s}^{(3)}|\right)\leq C(\compact_l)n^{-1}\PE_{x,\theta}^{(l)}\left(V^{2(\beta+\alpha\kappa_0)}(\tilde X_n)\un_{\{\tauout_{\compact_l}>n\}}\right)
=O(n^{1-2(\beta+\alpha\kappa_0)})$. The rest of the proof is similar to the above.
\end{proof}

\subsection{Proof of Proposition \ref{propcvSA2}}\label{proofpropcvSA2}

\begin{proof}
We will show that for any $p\geq 0$, $n\geq 1$, any compact subset $\compact$ of $\Theta$ and any $\delta>0$,
\begin{equation}\label{proofpropcvSA2TBS}\sup_{(x,\theta)\in\Xset_0\times \Theta_0}\PP_{x,\theta}^{(p)}\left(C_{n,p}(\compact)>\delta\right)\leq \mathcal{B}(n,p),\end{equation}
where the bound $\mathcal{B}(n,p)$ satisfies $\lim_{n\to\infty}\mathcal{B}(n,p)=0$ for any $p\geq 0$ and $\lim_{p\to\infty}\mathcal{B}(n,p)=0$ for any $n\geq 1$. This clearly implies (\ref{eqcvSA1}) and (\ref{eqcvSA2}) and the result will follow from Proposition \ref{propcvSA}. We have
\begin{equation}\label{eq1:proofpropcvSA2}
C_{n,p}(\compact)\leq \sup_{l\geq n}\un_{\{\tauout_{\compact}>l\}}\left|\sum_{j=n}^l\gamma_{p+j-1}\tilde\epsilon^{(1)}_{j}\right|\;+\;
\sup_{l\geq n}\un_{\{\tauout_{\compact}>l\}}\left|\sum_{j=n}^l\gamma_{p+j-1}\tilde \epsilon^{(2)}_{j}\right|.\end{equation}
We start with the second term on the rhs of (\ref{eq1:proofpropcvSA2}). By Doob's inequality and B\ref{B2}, for $N>n$,
\begin{multline*}
\PP_{x,\theta}^{(l)}\left(\sup_{n\leq l\leq N}\un_{\{\tauout_{\compact}>l\}}\left|\sum_{j=n}^l\gamma_{p+j-1}\tilde \epsilon^{(2)}_{j}\right|>\delta\right)\\
\leq \delta^{-2}\PE_{x,\theta}^{(l)}\left(\sum_{j=n}^N\gamma_{p+j-1}^2\un_{\{\tauout_{\compact}>j\}}
\int\Phi_{\tilde \theta_{j}}^2(\tilde X_{j},y)q^{(1)}_{\tilde \theta_{j}}(\tilde X_{j},dy)\right)\\
\leq C(\compact)\delta^{-2}\PE_{x,\theta}^{(l)}\left(\sum_{j=n}^N
\gamma_{p+j-1}^2\un_{\{\tauout_{\compact}>j\}}V^{2\eta}(\tilde X_{j})\right)\\
\leq C(\compact)\delta^{-2}\left(\gamma_{p+n}^2\PE_{x,\theta}^{l}\left(\un_{\{\tauout_{\compact}>n-1\}}V^{2\eta+\alpha}(\tilde X_{n})\right)
+\sum_{j=n}^N\gamma_{p+j}^2\right).\end{multline*}
It follows that
\begin{multline}\label{eq2:proofpropcvSA2}
\PP_{x,\theta}^{(l)}\left(\sup_{l\geq n}\un_{\{\tauout_{\compact}>l\}}\left|\sum_{j=n}^l\gamma_{p+j-1}\tilde \epsilon^{(2)}_{j}\right|>\delta\right)\leq C(\compact)\delta^{-p}\left(\gamma_{p+n}^2n^{2\eta+\alpha}
+\sum_{j\geq n}\gamma_{p+j}^2\right)
.\end{multline}

To deal with the first term on the rhs of (\ref{eq1:proofpropcvSA2}), we proceed as in the proof of Theorem \ref{thm1}. We consider the sequence $\{a_n,\;n\geq 0\}$ such that $a_n\propto n^{-\rho}$, $a_n\in (0,1/2]$ where $\rho\in(0,1)$ is as in the statement of the Proposition. For $1\leq n\leq l$ and $p\geq 0$, we introduce the partial sum
\[S_{n,l}(p,\compact)\eqdef \un_{\{\tauout_\compact>l\}}\sum_{j=n}^l\gamma_{p+j}\bar \Upsi_{\tilde \theta_{j}}(\tilde X_j).\]
where $\bar\Upsi_\theta(x)=\Upsi_\theta(x)-h(\theta)$. Under B\ref{B2}, $\Upsi_\theta$ admits an approximate Poisson equation $\tilde g_a$ for any $j\geq 1$ and we have $\bar \Upsi_{\tilde \theta_{j}}(\tilde X_j)=(1-a_j)^{-1}\tilde g_{a_j}(\tilde X_j,\tilde \theta_{j})-P_{\tilde \theta_{j}} \tilde g_{a_j}(\tilde X_j,\tilde \theta_{j})$.
Using this and following the same approach as in the proof of Theorem \ref{thm1}, we decompose $S_{n,l}(p,\compact)$ as
\[S_{n,l}(p,\compact)=T_{n,l}^{(1)}+T_{n,l}^{(2)}+T_{n,l}^{(3)}+T_{n,l}^{(4)}+T_{n,l}^{(5)}+T_{n,l}^{(6)}\]
where
\[T_{n,l}^{(1)}=\un_{\{\tauout_\compact>l\}}\sum_{j=n}^l\un_{\{\tauout_\compact>j\}}
\gamma_{p+j}\left((1-a_j)^{-1}-1\right)\tilde g_{a_j}(\tilde X_j,\tilde \theta_{j}).\]
\[T_{n,l}^{(2)}=\un_{\{\tauout_\compact>n\}}\gamma_{p+n} \tilde g_{a_{n}}(\tilde X_{n},\tilde \theta_{n})-\un_{\{\tauout_\compact>l\}}\gamma_{p+l} P_{\tilde \theta_{l}}\tilde g_{a_{l}}(\tilde X_{l},\tilde \theta_{l}).\]
\[T_{n,l}^{(3)}=\un_{\{\tauout_\compact>l\}}\sum_{j=n}^{l-1}\un_{\{\tauout_\compact>j+1\}}\gamma_{p+j+1}
\left(\tilde g_{a_{j+1}}(\tilde X_{j+1},\tilde \theta_{j+1})-\tilde g_{a_{j+1}}(\tilde X_{j+1},\tilde \theta_{j})\right).\]
\[T_{n,l}^{(4)}=\un_{\{\tauout_\compact>l\}}\sum_{j=n}^{l-1}\un_{\{\tauout_\compact>j\}}\left(\gamma_{p+j+1}-\gamma_{p+j}\right)
\tilde g_{a_{j+1}}(\tilde X_{j+1},\tilde \theta_{j}).\]
\[T_{n,l}^{(5)}=\un_{\{\tauout_\compact>l\}}\sum_{j=n}^{l-1}\un_{\{\tauout_\compact>j\}}\gamma_{p+j}
\left(\tilde g_{a_j+1}(\tilde X_{j+1},\tilde \theta_{j})-\tilde g_{a_j}(\tilde X_{j+1},\tilde \theta_{j})\right).\]
\[T_{n,l}^{(6)}=\un_{\{\tauout_\compact>l\}}\sum_{j=n}^{l-1}\un_{\{\tauout_\compact>j\}}
\gamma_{p+j}\left(\tilde g_{a_j}(\tilde X_{j+1},\tilde \theta_{j})-P_{\tilde \theta_{j}}\tilde g_{a_j}(\tilde X_{j},\tilde \theta_{j})\right).\]
We deal with each of these terms using similar techniques as in the proofs of Theorem \ref{thm1} and Theorem \ref{thm3}. Some of the details are thus omitted. Let $\delta>0$ arbitrary.

\paragraph{\tt On Term $T_{n,l}^{(1)}$}  Take $\kappa>1$ such that $\eta+\alpha\kappa<1-\alpha$. Then Proposition \ref{propboundga} yields $|\tilde g_{a_j}(\tilde X_j,\tilde \theta_{j})|\leq C(\compact)V^{\eta+\alpha\kappa}(\tilde X_j)$ on $\{\tilde \theta_{j}\in\compact\}$. Then by Markov's inequality, we have
\begin{multline}\label{proofpropcvSA2eq1}
\PP_{x,\theta}^{(p)}\left(\sup_{l\geq n}\left|T_{n,l}^{(1)}\right|>\delta\right)\leq \delta^{-1} \PE_{x,\theta}^{(p)}\left(\sum_{j\geq n}\un_{\{\tauout_\compact>j-1\}}
\gamma_{p+j-1}\left((1-a_j)^{-1}-1\right)\left|\tilde g_{a_j}(\tilde X_j,\tilde \theta_{j})\right|\right)\\
\leq \delta^{-1}C(\compact)V(x)\left(\gamma_{n+p}n^{1-\alpha-\rho}+\sum_{j\geq n}\gamma_{p+j}j^{-\rho}\right).\end{multline}
The last inequality uses Proposition \ref{lem:sumdrift} and Proposition \ref{prop:useful} (i).

\paragraph{\tt On Term $T_{n,l}^{(2)}$}  Let $\epsilon>0$, $\kappa>1$ such that $\epsilon\in (\rho,(1-\alpha)(\eta+\kappa\alpha)^{-1}-1)$. That is $(1+\epsilon)(\eta+\alpha\kappa)<1-\alpha$ and $\epsilon>\rho$. Then
\begin{multline}\label{proofpropcvSA2eq4}
\PP_{x,\theta}^{(p)}\left(\sup_{l\geq n}\left|T_{n,l}^{(2)}\right|>\delta\right)\\
\leq(2/\delta)^{1+\epsilon}\PE_{x,\theta}^{(p)}\left(\un_{\{\tauout_\compact>n\}}\gamma_{p+n}^{1+\epsilon}\left|\tilde g_{a_n}(\tilde X_{n},\tilde \theta_{n})\right|^{1+\epsilon} + \sum_{l\geq n}\gamma_{p+l}^{1+\epsilon}\un_{\{\tauout_\compact>l\}}\left|P_{\tilde \theta_{l}}\tilde g_{a_l}(\tilde X_{l},\tilde \theta_{l})\right|^{1+\epsilon}\right)\\
\leq (2/\delta)^{1+\epsilon}C(\compact)V(x)\left(\gamma_{p+n}^{1+\epsilon}n^{1-\alpha}+\sum_{j\geq n-1}\gamma_{p+j}^{1+\epsilon}\right).\end{multline}

\paragraph{\tt On Term $T_{n,l}^{(3)}$}
 Take $\kappa>1$ and $\delta>0$ such that $2\eta+\alpha(\kappa+\delta)<1-\alpha$ and $\eta+\alpha(\kappa+\delta)<1/2$. By Proposition \ref{propboundgaga} and B\ref{B2}
$\left|\tilde g_a(x,\theta)-\tilde g_a(x,\theta')\right|\leq C(\compact)\sup_{\theta\in\compact}|\Upsi_\theta|_{V^\eta}\zeta_\delta(a)\left|\theta-\theta'\right|V^{\eta+\alpha(\kappa+\delta)}(x)$. Then by Markov's inequality
\begin{multline*}
\PP_{x,\theta}^{(p)}\left(\sup_{\geq n}\left|T_{n,l}^{(3)}\right|>\delta\right)
\leq (1/\delta)\PE_{x,\theta}^{(p)}\left(\sum_{j\geq n}\un_{\{\tauout_\compact>j\}}\gamma_{p+j+1}^2\zeta_\delta(a_j)
\left|\Phi_{\tilde \theta_j}(\tilde X_j,Y_{j+1})\right|V^{\eta+\alpha(\kappa+\delta)}(\tilde X_{j+1})\right).\end{multline*}
From B\ref{B2} and the structure of the algorithm we compute that
\[\PE_{x,\theta}^{(p)}\left(\left|\Phi_{\tilde \theta_j}(\tilde X_j,Y_{j+1})\right|V^{\eta+\alpha(\kappa+\delta)}(\tilde X_{j+1})|\F_j\right)
\un_{\{\tauout_{\compact}>j\}}\leq C(\compact)V^{2\eta+\alpha(\kappa+\delta)}(\tilde X_j).\]
It follows
\begin{multline}\label{proofpropcvSA2eq6}
\PP_{x,\theta}^{(p)}\left(\sup_{\geq n}\left|T_{n,l}^{(3)}\right|>\delta\right)
\leq(1/\delta)C(\compact)\left(\gamma_{p+n-1}^2n^{1+\rho-\alpha}
+\sum_{j\geq n}\gamma_{p+j-1}^2j^{\rho}\right)V(x).\end{multline}

\paragraph{\tt On Term $T_{n,l}^{(4)}$}  By Markov's inequality,
\begin{multline}\label{proofpropcvSA2eq5}
\PP_{x,\theta}^{(p)}\left(\sup_{l\geq n}\left|T_{n,l}^{(4)}\right|>\delta\right)
\leq(1/\delta)\PE_{x,\theta}^{(p)}\left(\sum_{j\geq n}\left(\gamma_{p+j}-\gamma_{p+j+1}\right)
\un_{\{\tauout_\compact>j\}}\left|\tilde g_{a_{j+1}}(\tilde X_{j+1},\tilde \theta_{j})\right|\right)\\
\leq (1/\delta)C(\compact)\PE_{x,\theta}^{(p)}\left(\sum_{j\geq n}\left(\gamma_{p+j}-\gamma_{p+j+1}\right)
\un_{\{\tauout_\compact>j\}}V^{1-\alpha}(\tilde X_{j+1})\right)\\
\leq(1/\delta)C(\compact)V(x)\left(n^{1-\alpha}(\gamma_{p+n}-\gamma_{p+n+1})
+\gamma_{p+n}\right).\end{multline}

\paragraph{\tt On Term $T_{n,l}^{(5)}$}   Take $\kappa\in (1,2)$ such that $\eta+\alpha\kappa<1-\alpha$. One can check as in Proposition \ref{propboundgaga} that for any compact $\compact$ $\left|P_\theta \tilde g_a(x,\theta)-P_\theta \tilde g_{a'}(x,\theta)\right|\leq C(\compact)|a-a'|a^{\kappa-2}V^{\eta+\alpha\kappa}(x)$. And for $a_j\propto j^{-\rho}$, $|a_j-a_{j-1}|a_{j}^{\kappa-2}\propto j^{-1}a_j^{\kappa-1}=o(j^{-1})$. Hence, by  Markov's inequality, we get:
\begin{equation}\label{proofpropcvSA2eq3}
\PP_{x,\theta}^{(p)}\left(\sup_{l\geq n}\left|T_{n,l}^{(5)}\right|>\delta\right)
\leq \delta^{-1}C(\compact)V(x)\left(n^{-\alpha}\gamma_{p+n}+\sum_{j\geq n}\gamma_{p+j}j^{-1}\right).\end{equation}

\paragraph{\tt On Term $T_{n,l}^{(6)}$}  Let $\kappa>1$ such that $2(\eta+\alpha\kappa/2)<1-\alpha$. Consider the term  $D_j=\un_{\{\tauout_\compact>j\}}
\gamma_{p+j}\left(\tilde g_{a_j}(\tilde X_{j+1},\tilde \theta_{j})-P_{\tilde \theta_{j}}\tilde g_{a_j}(\tilde X_{j},\tilde \theta_{j})\right)$ so that $T_{n,l}^{(2)}=\un_{\{\tauout_\compact>l\}}\sum_{j=n}^{l-1}D_j$. We note that $D_j$ is a martingale difference and by Doob's inequality we get:
\begin{multline}\label{proofpropcvSA2eq2}
\PP_{x,\theta}^{(p)}\left(\sup_{l\geq n}\left|T_{n,l}^{(6)}\right|>\delta\right)\leq (1/\delta)^2\sum_{j\geq n}\PE_{x,\theta}^{(l)}\left(|D_j|^2\right)\\
\leq (1/\delta)^2C(\compact)V(x)\left(\gamma_{p+n-1}^2n^{1-\alpha+\rho}+ \sum_{j\geq n}\gamma_{p+j-1}^{2}j^\rho\right).\end{multline}

By combining (\ref{eq2:proofpropcvSA2})-(\ref{proofpropcvSA2eq2}) and (\ref{eq1thm3}), we get (\ref{proofpropcvSA2TBS}) as claimed.
\end{proof}

\subsection{Proof of the results of Section \ref{sec:ex}}\label{proofEx}

\subsubsection{Proof of Proposition \ref{proplyapEx}}\label{proofproplyapEx}

The function $a(\theta)$ is of class $\mathcal{C}^1$. Hence by Assumption C\ref{C1} and the Mean Value Theorem $\L=\{\theta\in\rset:\; a(\theta)=\bar\alpha\}$ is not empty. It also follows from C\ref{C1} that the function $\theta\to \int_0^\theta\cosh(u)(\bar\alpha-a(u))du$ is bounded from below; so we can find $K_1$ such that $w(\theta)=\int_0^\theta\cosh(u)(\bar\alpha-a(u))du + K_1\geq 0$. Moreover $(a(u)-\bar\alpha)w'(\theta)=-\cosh(\theta)(a(\theta)-\bar\alpha)^2\leq 0$ with equality iif $\theta\in\L$. By Sard's theorem $w(\L)$ has an empty interior. Again from C\ref{C1}, it follows that $\L$ is included in a bounded interval of $\rset$ and since $\lim_{\theta\to\pm\infty}w(\theta)=\infty$, we can find $M_0$ such that $\L\subset\{\theta\in\rset:\; w(\theta)<M_0\}$ and $\W_{M}$ is bounded thus compact for any $M>0$.

\subsubsection{Proof of Proposition \ref{propstabEx}}\label{proofpropstabEx}

A straightforward calculation using the boundedness of $|\nabla\log\pi(x)|$ implies that for any $\theta\in\compact$,
\[\left|\frac{\partial}{\partial \theta}\log\left(\alpha_\theta(x,y)q_\theta(x,y)\right)\right|\leq C(\compact)\left(1+|y-x|^2\right),\]
for some finite constant $C(\compact)$. It follows that
\[\int\left|\frac{\partial}{\partial \theta}\left(\alpha_\theta(x,y)q_\theta(x,y)\right)f(y)\right|dy\leq C(\compact)|f|_{V_s^\beta}\int
\left(1+|y-x|^2\right)V_s^\beta(y)q_\theta(x,y)dy.\]
We do a change of variable $y=b(x)+e^{\theta/2}z$, where $b(x)=x+0.5e^{\theta}\nabla\log\pi(x)$ and using the boundedness of $|\nabla\log\pi(x)|$, we get:
\[\sup_{\theta\in\compact}\int\left|\frac{\partial}{\partial \theta}\left(\alpha_\theta(x,y)q_\theta(x,y)\right)f(y)\right|dy\leq C(\compact)|f|_{V_s^\beta} V_s^\beta(x)\int
\left(1+|z|^2\right)^{\beta s/2}g(z)dz,\]
where $g$ is the density of the mean zero $d$-dimensional Gaussian distribution with covariance matrix $I_d$. The stated result follows by an application of the Mean Value Theorem.


\vspace{2cm}

{\bf Acknowledgment:} We are grateful to Michael Woodroofe for helpful
discussions on martingale approximation techniques and to Shukri Osman for helpful conversations.

\vspace{1cm}

\bibliographystyle{ims}
\bibliography{af2}

\end{document}